\documentclass[12pt,reqno]{amsart}
\ifx\pdfsuppressptexinfo\undefined\else\pdfsuppressptexinfo=-1\fi
\usepackage{amsmath,amsthm,amssymb}
\usepackage{enumitem}
\usepackage{array}
\usepackage{needspace}
\usepackage{placeins}
\usepackage[hmargin=1.15in,vmargin=1.15in]{geometry}
\usepackage[colorlinks,linkcolor=blue,citecolor=blue,urlcolor=blue]{hyperref}
\usepackage{bookmark}
\setlist[enumerate,1]{label=(\arabic*),ref=\arabic*,font=\normalfont}
\setlist[enumerate,2]{label=(\alph*),font=\normalfont}
\numberwithin{equation}{section}
\newcommand{\F}{\mathbb F}
\newcommand{\PP}{\mathbb P}
\DeclareMathOperator{\Aut}{Aut}
\DeclareMathOperator{\GL}{GL}
\DeclareMathOperator{\PGL}{PGL}
\DeclareMathOperator{\SL}{SL}
\DeclareMathOperator{\PSL}{PSL}
\theoremstyle{plain}
\newtheorem{theorem}{Theorem}[section]
\newtheorem{proposition}[theorem]{Proposition}
\newtheorem{lemma}[theorem]{Lemma}
\newtheorem{corollary}[theorem]{Corollary}
\newtheorem*{theoremA}{Theorem~A}
\newtheorem*{theoremB}{Theorem~B}
\newtheorem*{theoremC}{Theorem~C}
\theoremstyle{remark}
\newtheorem{remark}[theorem]{Remark}

\begin{document}
\title[Maximal curves not covered by the Hermitian curve]
{A new family of maximal curves not covered by the Hermitian curve}
\date{}
\author[Liming Ma]{Liming Ma}
\address{School of Mathematical Sciences, University of Science and Technology of China,
Hefei, Anhui 230026, China}
\email{lmma20@ustc.edu.cn}
\author[Yipeng Wang]{Yipeng Wang}
\address{School of Mathematical Sciences, University of Science and Technology of China,
Hefei, Anhui 230026, China}
\email{jisunxingfu@mail.ustc.edu.cn}
\hypersetup{
  pdftitle={A new family of maximal curves not covered by the Hermitian curve},
  pdfauthor={Liming Ma and Yipeng Wang}}
\begin{abstract}
For every prime power $q$ and integer $n\ge1$, we give a single
Kummer equation defining an $\F_{q^{2n}}$-maximal curve
$\mathcal X_{q,n}$.
The family recovers the Hermitian curves for $n=1,2$, the
Giulietti--Korchm\'aros curve for $n=3$, and the
Beelen--Montanucci curves for odd $n\ge5$.
For $q>2$ and even $n\ge4$, it gives new maximal curves of genus
$(q^2-1)q^n/2$ that are not covered by the Hermitian curve over
$\F_{q^{2n}}$.
We compute the genus for both parities of $n$ and give a uniform
proof of maximality by counting rational places.
For $q>2$ and odd $n\ge5$, we remove the Galois hypothesis
from the known non-covering results for the Beelen--Montanucci
and Garcia--G\"uneri--Stichtenoth curves.
For even $n\ge4$, we determine the full automorphism group,
of order $(q^n+1)q(q^2-1)$, and give
explicit equations and genera for all Galois quotients,
with closed-form equations for several classes of subgroups.
\end{abstract}
\maketitle

\section{Introduction}
\label{b1:introduction}

Let \(l\) be a prime power. A nonsingular, geometrically irreducible
projective curve \(\mathcal X\) over \(\F_{l^2}\) is called maximal
if its number of rational points attains the Hasse--Weil upper bound
\(\#\mathcal X(\F_{l^2})=l^2+1+2g(\mathcal X)l\).
The possible genera of maximal curves have been studied extensively;
see, for example, \cite{CHKT2000,FGT1997,KT2002,XS1995}.
An important source of examples is the Hermitian curve
\(\mathcal H_l\), with affine equation \(b^l+b=a^{l+1}\), which is
maximal over \(\F_{l^2}\) and has genus \(l(l-1)/2\).
R\"uck and Stichtenoth proved that, up to \(\F_{l^2}\)-isomorphism,
it is the unique \(\F_{l^2}\)-maximal curve of this genus
\cite{RS1994}.
A theorem of Serre states that a curve covered over \(\F_{l^2}\)
by an \(\F_{l^2}\)-maximal curve is also maximal over \(\F_{l^2}\);
see \cite[p.~230]{GK2009}.
Thus subcovers of \(\mathcal H_l\) provide a systematic way of
constructing maximal curves; see, for example, \cite{CKT2000,GSX2000}.

For many years, it was not known whether every maximal curve over
\(\F_{l^2}\) was covered by \(\mathcal H_l\) over the same field.
Garcia and Stichtenoth constructed an \(\F_{27^2}\)-maximal curve
of genus \(24\) that is not a Galois subcover of \(\mathcal H_{27}\)
\cite{GS2006}.
Giulietti and Korchm\'aros settled the question for arbitrary,
not necessarily Galois, coverings by constructing the GK curve.
For a prime power \(q\), this curve has the affine model
\cite[Remark~2.1]{GGS2010}
\begin{equation}
\label{b1:gk-model}
 x^q+x=y^{q+1},\qquad w^{q^2-q+1}=y^{q^2}-y.
\end{equation}
It is maximal over \(\F_{q^6}\) and, for \(q>2\), is not
covered by \(\mathcal H_{q^3}\) over \(\F_{q^6}\)
\cite[Theorems~1 and~5]{GK2009}.

Other examples not covered by the Hermitian curve were obtained
as GK subcovers in \cite{TTT2016}.
In a recent preprint, Beelen, Montanucci, and Niemann constructed
the first known such example over \(\F_{p^2}\) with \(p\) prime: an \(\F_{11^2}\)-maximal
curve of genus \(17\) \cite[Theorems~2.3 and~3.1]{BMN2026}.

Further examples not Galois covered by the Hermitian curve come
from Skabelund's maximal cyclic covers of the Suzuki and Ree
curves \cite{Skabelund2018}; their Galois non-covering was proved
by Giulietti, Montanucci, Quoos, and Zini \cite{GMQZ2018}.

For odd integers \(n\ge3\), put \(m=(q^n+1)/(q+1)\).
Garcia, G\"uneri, and Stichtenoth extended the GK construction to
\(\F_{q^{2n}}\) by the affine equations
\cite[equation~(2.1)]{GGS2010}
\begin{equation}
\label{b1:intro-ggs-model}
 x^q+x=y^{q+1},\qquad w^m=y^{q^2}-y.
\end{equation}
Beelen and Montanucci constructed a second family over the same
fields, with affine equations
\cite[equation~(2.1)]{BM2018}
\begin{equation}
\label{b1:bm-model}
 x^{q+1}-y^{q+1}=1,\qquad w^m=\frac{x^{q^2}-x}{y^q}.
\end{equation}
We call these the GGS and BM families, respectively. Both are
maximal over \(\F_{q^{2n}}\) and have the same genus for each odd
\(n\). For \(n=3\), both are \(\F_{q^6}\)-isomorphic to the
GK curve; see \cite[Remark~2.1]{GGS2010} and
\cite[Section~2]{BM2018}.
For odd \(n\ge5\), they are not isomorphic over an algebraic closure
\cite[Corollary~2.6 and Remark~4.5]{BM2018}.
The exponent \(m=(q^n+1)/(q+1)\) in the GGS and BM
constructions is integral only for odd \(n\).

For \(q>2\) and odd \(n\ge5\), Duursma and Mak proved that the
GGS curves are not Galois subcovers of \(\mathcal H_{q^n}\) over
\(\F_{q^{2n}}\) \cite[Theorem~1.1]{DM2012}.
Giulietti, Montanucci, and Zini treated the remaining case
\(q=2\) \cite[Theorem~1.1]{GMZ2016}.
For \(q>2\) and odd \(n\ge3\), Beelen and Montanucci proved the
corresponding Galois non-covering result for the BM curves
\cite[Corollary~2.2]{BM2018}.
For odd \(n\ge5\), these results leave open the possibility of
non-Galois coverings.
A maximal curve covered but not Galois covered by the Hermitian
curve was constructed in \cite{MW2026}. In a recent preprint,
Bartoli, Borges, Timpanella, and Zini constructed infinite families
in every odd characteristic \cite[Theorem~1.4]{BBTZ2026}.

In this paper we construct a Kummer family of \(\F_{q^{2n}}\)-maximal
curves for every prime power \(q\) and every integer \(n\ge1\).
For \(q>2\) and even \(n\ge4\), it gives new maximal curves.
We also settle the question of arbitrary coverings for the BM and
GGS families when \(q>2\) and \(n\ge5\) is odd.
Let \(q\) be a prime power, let \(n\ge1\), and put
\(k=\F_{q^{2n}}\). Choose \(c\in k^*\) with \(c^{q^n}=-c\).
Let \(\mathcal X_{q,n}\) be the nonsingular projective curve over
\(k\) with function field \(k(u,z)\) given by
\begin{equation}
\label{b1:intro-model}
 z^{q^n+1}=c(u^q-u)
 \bigl(1+(u^q-u)^{q-1}\bigr)^{\frac{q^{n-1}-1}{q-1}}.
\end{equation}
Its \(k\)-isomorphism class is independent of the choice of \(c\).
For \(n=1,2\), the curve \(\mathcal X_{q,n}\) is
\(k\)-isomorphic to \(\mathcal H_{q^n}\); for \(n=3\), it is
\(k\)-isomorphic to the GK curve.
For odd \(n\ge5\), it is \(k\)-isomorphic to the BM curve.
These identifications are given in
Remark~\ref{b1:small-members} and
Section~\ref{b1:known-members-subsection}.

We write \(\Aut(\mathcal X)\) for the automorphism group of a curve
\(\mathcal X\) over an algebraic closure \(\overline{k}\) of \(k\).
Our first main result is the following.

\Needspace{10\baselineskip}
\begin{theoremA}
Let \(q\) be a prime power and let \(n\ge1\).
\begin{enumerate}
\item The curve \(\mathcal X_{q,n}\) is maximal over \(\F_{q^{2n}}\), and
its genus is
\[
 g(\mathcal X_{q,n})=
 \begin{cases}
 \bigl((q^2-1)q^n-q^2(q-1)\bigr)/2,&n\text{ odd},\\
 (q^2-1)q^n/2,&n\text{ even}.
 \end{cases}
\]
\item If \(q>2\) and \(n\ge3\), there is no nonconstant
\(\F_{q^{2n}}\)-morphism \(\mathcal H_{q^n}\to\mathcal X_{q,n}\).
\item If \(n\ge4\), then \(\Aut(\mathcal X_{q,n})\) has order
\((q^n+1)q(q^2-1)\), and all its elements are defined over
\(\F_{q^{2n}}\).
\end{enumerate}
\end{theoremA}

For \(n=1,2\), Theorem~A(1) recovers the genus and maximality
of the Hermitian curve \(\mathcal H_{q^n}\).
For \(n=3\), parts~(1) and~(2) recover the results for the GK
curve \cite[Theorems~1,~2, and~5]{GK2009}.
For odd \(n\ge5\), the genus and maximality in part~(1) and the
automorphism group in part~(3) are due to Beelen and Montanucci
\cite[Proposition~2.1 and Theorems~3.9 and~4.3]{BM2018}.
Our proof of maximality treats both parities of \(n\) uniformly.
Proposition~\ref{b1:group-structure} gives the group structure
for even \(n\). For \(n=3\) and \(q>2\), the GK curve has the
larger automorphism group of order
\(q^3(q^3+1)(q^2-1)(q^2-q+1)\)
\cite[Theorem~6]{GK2009}.

For each fixed \(q>2\), varying the even integer \(n\ge4\)
gives infinitely many new maximal curves not covered by
\(\mathcal H_{q^n}\) over \(\F_{q^{2n}}\).
The restriction \(q>2\) is necessary: for \(q=2\), the curve
\(\mathcal X_{2,n}\) is a Galois subcover of \(\mathcal H_{2^n}\)
for every \(n\ge1\); see Proposition~\ref{b1:q2-cover}.

For the two generalizations of the GK curve, we obtain the following.

\Needspace{5\baselineskip}
\begin{theoremB}
Let \(q>2\) and let \(n\ge5\) be odd. Neither the
Beelen--Montanucci curve nor the Garcia--G\"uneri--Stichtenoth
curve is covered by \(\mathcal H_{q^n}\) over \(\F_{q^{2n}}\).
\end{theoremB}

The BM assertion follows from Theorem~A(2) and
Proposition~\ref{b1:bm-identification}; the GGS assertion is
Theorem~\ref{b1:ggs-noncover}. These results remove the Galois
hypothesis from \cite[Corollary~2.2]{BM2018} and
\cite[Theorem~1.1]{DM2012}, respectively.

In odd characteristic, the genera of Hermitian quotients by subgroups
fixing a rational point were determined in \cite{BMXY2013}.
The complete genus spectrum of Hermitian quotients for
\(q\equiv1\pmod4\) was determined in \cite{MZ2020}.
Explicit equations for all Hermitian quotients by subgroups of
order \(p^2\), where \(p\) is the characteristic, were determined
in \cite{GK2024}.
The genera of many Galois quotients of the GK curve were computed in
\cite{FG2012}, and explicit equations for several quotient families
were obtained in \cite{GQZ2016}.
Complete descriptions of the automorphism subgroups and the genera
of their fixed fields are known for the GGS curves with odd
\(n\ge5\) \cite{ABB2017} and for the Skabelund curves
\cite{BLM2023}.
Galois subfields of the BM curves for odd \(n\) have been studied in
\cite{BDMZ2026,BM2020,MQ2022}.

For even \(n\ge4\), we determine explicit generators, defining
equations, and genera for every Galois quotient of \(\mathcal X_{q,n}\).
Write \(F=k(u,z)\) for the function field of \(\mathcal X_{q,n}\),
and let \(C\) be the cyclic group of order \(q^n+1\) fixing \(u\).
For \(S\le\PGL_2(q)\), let \(\nu(S)\) be the number of elements
of order \(p\) in \(S\), where \(p\) is the characteristic.

\Needspace{6\baselineskip}
\begin{theoremC}
Suppose that \(n\ge4\) is even. For every subgroup
\(U\le\Aut(\mathcal X_{q,n})\), put \(d=|U\cap C|\),
\(\mu=(q^n+1)/d\), and let
\(S\le\PGL_2(q)\) be its image under the induced action on \(k(u)\).
\begin{enumerate}
\item There are functions \(t,v\in F\), given in
\eqref{b1:quotient-rational-generator} and
Theorem~\ref{b1:even-quotient-equations}, monic polynomials
\(\Phi_S,\Psi_S\in\F_{q^4}[T]\), and a constant \(\kappa_S\in k^*\)
such that
\begin{equation}
\label{b1:intro-quotient-model}
 F^U=k(t,v),\qquad
 v^\mu=\kappa_S\Phi_S(t)
 \Psi_S(t)^{\frac{q^{n-1}-1}{q-1}}.
\end{equation}
Here \(k(t)=k(u)^S\), \(\deg\Phi_S=q+1\), and
\(\deg\Psi_S=q^2-q\).
\item The genus of \(F^U\) is
\begin{equation}
\label{b1:intro-quotient-genus}
 g(F^U)=\frac{q^2-1-\nu(S)}{2|S|}(\mu-1).
\end{equation}
\end{enumerate}
\end{theoremC}

Proposition~\ref{b1:subgroup-realization} gives a complete
parameterization of the subgroups in terms of \((d,S)\), and
Theorem~\ref{b1:even-quotient-equations} gives their fixed fields
by the formulas \eqref{b1:quotient-rational-generator}--\eqref{b1:even-quotient-model}.
For fixed \(q\) and \(S\), the polynomials
\(\Phi_S,\Psi_S\) can be chosen independently of \(n\).
Together with the subgroup list of \(\PGL_2(q)\),
\eqref{b1:intro-quotient-genus} gives the complete
genus spectrum of the Galois quotients of \(\mathcal X_{q,n}\)
for each even \(n\ge4\).
Section~\ref{b1:explicit-quotient-families} gives closed-form
models for affine, cyclic, and dihedral images, and for
\(\PGL_2(Q)\) and \(\PSL_2(Q)\) when \(q=Q^2\).

We prove maximality by counting rational places that split completely
in \(k(u,z)/k(u)\), using fractional linear transformations over
\(\F_q\) and the norm from \(k\) to \(\F_{q^n}\).
The argument applies to both parities of \(n\) and gives
independent proofs of maximality for the GK and BM curves.

For the GK curve and certain GK subcovers, non-covering follows
by comparing degree bounds from rational point counts and the
Hurwitz genus formula \cite{GK2009,TTT2016}. For the BM and GGS
curves, these bounds need not exclude all possible degrees.
Our proof uses the Hurwitz genus formula to bound the covering
degree \(d\). Lemma~\ref{b1:polynomial-representation} represents
pullbacks of functions in a Riemann--Roch space
\(\mathcal L((q^n+1)P)\) at a rational point \(P\) as ratios of
homogeneous polynomials of degree \(d\). After clearing denominators,
the resulting relations have degree less than \(q^n+1\), so they
are polynomial identities.
For odd \(n\), both the BM and GGS function fields contain
functions \(x,y\in\mathcal L((q^n+1)P)\), with \(x\) nonconstant,
satisfying \(x^q+x=y^{q+1}\). The resulting polynomial identity
would make \(\mathcal H_q\) rational
(Lemma~\ref{b1:hermitian-relation-criterion}).
For even \(n\), the identities for \(\mathcal X_{q,n}\) force an
irreducible factor to have multiplicity greater than the degree
of the polynomial it divides. Neither argument requires a
Galois covering.

This paper is organized as follows. Section~\ref{b1:model-section}
constructs the family, computes its genus, proves maximality,
and identifies its Hermitian, GK, and BM members.
Section~\ref{b1:noncover-section} excludes coverings by the Hermitian
curve for \(\mathcal X_{q,n}\) and for the GGS curves.
Section~\ref{b1:automorphisms-section} determines the full
automorphism group. Section~\ref{b1:quotient-equations-section}
gives all Galois subfields for even \(n\), with their equations
and genera. Section~\ref{b1:explicit-quotient-families} derives
closed-form models for the subgroup types described above.

\section{A new family of maximal curves}
\label{b1:model-section}

In this section, we provide a unified explicit defining equation for a family of maximal curves including the Hermitian curves, GK curves and BM curves. 

We retain the notation of the introduction and write \(p\) for the
characteristic of \(k\). We use the notation and terminology of
\cite{Stichtenoth2009} for algebraic function fields.
All isomorphisms and morphisms are over \(k\) unless another field
is specified.

\subsection{Construction and genus}
\label{b1:ramification-subsection}

Put \(N=q^n+1\), \(r=(q^{n-1}-1)/(q-1)\),
\(A(T)=T^q-T\), \(B(T)=1+(T^q-T)^{q-1}\), and
\(f(T)=cA(T)B(T)^r\). We have
\(B(T)=(T^{q^2}-T)/(T^q-T)\) and
\(q+1+(q^2-q)r=N\).
Throughout, put \(\ell_0=\lfloor r/(q+1)\rfloor\).
Then \(r=(q+1)\ell_0\) for odd \(n\), and
\(r=(q+1)\ell_0+1\) for even \(n\).
For odd \(n\), we retain \(m=N/(q+1)\).
We write the function field of \(\mathcal X_{q,n}\) as
\begin{equation}
\label{b1:model}
 F=k(u,z),\qquad z^{q^n+1}=cA(u)B(u)^r.
\end{equation}

The factors \(A\) and \(B\) have simple zeros precisely in
\(\F_q\) and \(\F_{q^2}\setminus\F_q\), respectively.
The right-hand side of \eqref{b1:model} has a simple zero at \(u=0\).
Eisenstein's criterion at \(u=0\), applied over \(\overline k\),
shows that the defining polynomial is absolutely irreducible
and \([F:k(u)]=q^n+1\).
As \(q^n+1\mid q^{2n}-1\), all \((q^n+1)\)-st roots of unity lie in \(k\),
and \(F/k(u)\) is cyclic of degree \(q^n+1\).

The trace map from \(k\) to \(\F_{q^n}\) has a nonzero element in its
kernel, so the choice of \(c\) is possible.  If \(c'\) is another such
choice, then \(c'/c\in\F_{q^n}^*\).  Surjectivity of the norm
\(k^*\to\F_{q^n}^*\) gives \(\lambda\in k^*\) with
\(\lambda^{q^n+1}=c'/c\), and replacing \(z\) by \(\lambda z\)
gives the model with constant \(c'\).
Thus the \(k\)-isomorphism class of \(\mathcal X_{q,n}\) is independent
of the choice of \(c\).

\begin{remark}
\label{b1:small-members}
For \(n=1,2\), equation \eqref{b1:model} reduces to
\(z^{q^n+1}=c(u^{q^n}-u)\). The substitutions \(a=z\) and
\(b=-cu\) give \(b^{q^n}+b=a^{q^n+1}\), so
\(\mathcal X_{q,n}\) is \(k\)-isomorphic to \(\mathcal H_{q^n}\).
\end{remark}

We compute the ramification in \(F/k(u)\). For \(a\in k\), let
\(Q_a\) denote the zero of \(u-a\) in \(k(u)\), and let
\(Q_\infty\) denote the pole of \(u\) in \(k(u)\).
The distinction between odd and even \(n\) occurs at the zeros of \(B\). Write
\(\varepsilon=\gcd(r,q^n+1)\).
The identity \(q^n+1=q(q-1)r+q+1\) gives
\(\gcd(r,q^n+1)=\gcd(r,q+1)\). Reducing
\(r=\sum_{i=0}^{n-2}q^i\), with the empty sum equal to zero,
modulo \(q+1\) gives zero for odd \(n\) and one for even \(n\).
Thus \(\varepsilon=q+1\) when \(n\) is odd,
and \(\varepsilon=1\) when \(n\) is even.

\Needspace{5\baselineskip}
\begin{lemma}
\label{b1:ramification}
The following hold.
\begin{enumerate}
\item\label{b1:ramification-indices}
Every place of \(k(u)\) other than \(Q_\infty\) and \(Q_a\),
\(a\in\F_{q^2}\), is unramified in \(F\). The \(q+1\) places
\(Q_\infty\) and \(Q_a\), \(a\in\F_q\), are totally ramified.
At the remaining \(q^2-q\)
places the ramification index is \((q^n+1)/(q+1)\) when \(n\) is
odd, and \(q^n+1\) when \(n\) is even.
\item\label{b1:infinite-place}
The unique place \(P_\infty\) of \(F\) lying above \(Q_\infty\)
is rational, and
\((u)_\infty=(q^n+1)P_\infty\), \((z)_\infty=q^nP_\infty\).
\end{enumerate}
\end{lemma}

\begin{proof}
For every place \(Q\) of \(k(u)\),
\begin{equation}
\label{b1:kummer-valuations}
 \operatorname{ord}_{Q}\!\bigl(cA(u)B(u)^r\bigr)=
 \begin{cases}
  1,&Q=Q_a,\ a\in\F_q,\\
  r,&Q=Q_a,\ a\in\F_{q^2}\setminus\F_q,\\
  -q^n,&Q=Q_\infty,\\
  0,&\text{otherwise}.
 \end{cases}
\end{equation}
Since \(p\nmid q^n+1\),
\cite[Proposition~3.7.3]{Stichtenoth2009} gives the asserted
ramification indices and shows that every other place is unramified.
The place \(Q_\infty\) is rational and totally ramified, so \(P_\infty\) is
rational. The valuations of \(u\) and \(z\) at \(P_\infty\) follow
from \eqref{b1:model}.
\end{proof}

\begin{proposition}
\label{b1:genus}
The genus of \(\mathcal X_{q,n}\) is
\begin{equation}
\label{b1:genus-formula}
 g(\mathcal X_{q,n})=
 \begin{cases}
 \bigl((q^2-1)q^n-q^2(q-1)\bigr)/2,&n\text{ odd},\\
 (q^2-1)q^n/2,&n\text{ even}.
 \end{cases}
\end{equation}
\end{proposition}

\begin{proof}
Applying \cite[Corollary~3.7.4]{Stichtenoth2009} with the valuations
in \eqref{b1:kummer-valuations} gives
\[
 2g(F)-2=-2(q^n+1)+(q+1)q^n+(q^2-q)(q^n+1-\varepsilon).
\]
Substituting \(\varepsilon=q+1\) for odd \(n\) and
\(\varepsilon=1\) for even \(n\) gives
\eqref{b1:genus-formula}.
\end{proof}

\begin{remark}
\label{b1:alternative-even-model}
For even \(n\ge2\), put \(e=(q^n-q^2)/(q+1)\).
Then \(e=q(q-1)\ell_0\) is even and
\(r(e+1)-\ell_0(q^n+1)=1\).
The mutually inverse substitutions
\[
 y=\frac{z^{e+1}}{B(u)^{\ell_0}},\qquad
 z=\frac{y^r}{(cA(u))^{\ell_0}}
\]
give the equivalent model
\begin{equation}
\label{b1:alternative-even-equation}
 F=k(u,y),\qquad
 y^{q^n+1}=\eta(u^{q^2}-u)(u^q-u)^e,
\end{equation}
where \(\eta=c^{e+1}\), so \(\eta^{q^n}=-\eta\).
The additive factor \(u^{q^2}-u\) also occurs in the GK model
\eqref{b1:gk-model}. For \(n=2\), this equation specializes to
the Hermitian model \(y^{q^2+1}=c(u^{q^2}-u)\).
\end{remark}

\subsection{Maximality}
\label{b1:maximality-subsection}

The following identities describe the behavior of \(f\) under
fractional linear transformations over \(\F_q\).

\begin{lemma}
\label{b1:mobius-identities}
For \(M=\left(\begin{smallmatrix}\alpha&\beta\\\gamma&\delta\end{smallmatrix}\right)
\in\GL_2(q)\), write \(M(T)=(\alpha T+\beta)/(\gamma T+\delta)\).  Then
\begin{equation}
\label{b1:mobius-f}
 f(M(T))=\frac{\det(M)f(T)}{(\gamma T+\delta)^{q^n+1}}.
\end{equation}
Moreover, \(f(a)^{q^n}=-f(a^{q^n})\) for every \(a\in k\).
\end{lemma}

\begin{proof}
Using \(B(T)=(T^{q^2}-T)/A(T)\), substitution gives
\begin{equation}
\label{b1:mobius-ab}
 A(M(T))=\frac{\det(M)A(T)}{(\gamma T+\delta)^{q+1}},\qquad
 B(M(T))=\frac{B(T)}{(\gamma T+\delta)^{q^2-q}}.
\end{equation}
Since \(q+1+(q^2-q)r=q^n+1\), these identities give
\eqref{b1:mobius-f}. The equality \(f(a)^{q^n}=-f(a^{q^n})\)
follows from \(A,B\in\F_q[T]\) and \(c^{q^n}=-c\).
\end{proof}

We prove maximality by counting rational places of \(k(u)\) that split
completely in \(F/k(u)\). Recall that \(z^N=f(u)\).
The map \(\lambda\mapsto\lambda^{q^n+1}\) is the norm map
from \(k^*\) onto \(\F_{q^n}^*\), with kernel of order \(q^n+1\).
Hence, for \(a\in k\setminus\F_{q^2}\), the place \(Q_a\) splits
completely in \(F/k(u)\) if and only if \(f(a)\in\F_{q^n}^*\).

To find completely split places, we consider equations
\(a^{q^n}=M(a)\) with \(a\in k\setminus\F_{q^2}\) and
\(M\in\PGL_2(q)\). A trace-zero matrix representing \(M\) ensures
that such a solution satisfies \(f(a)^{q^n}=f(a)\), as shown in the
proof of Theorem~\ref{b1:maximality}.
Let \(\mathcal I_q\) be the set of elements of \(\PGL_2(q)\)
represented by matrices of trace zero in \(\GL_2(q)\).
For odd \(q\), these are precisely
the nonidentity involutions in \(\PGL_2(q)\). For even \(q\), scalar
matrices also have trace zero, so \(\mathcal I_q\) consists of the
nonidentity involutions and the identity.

\begin{lemma}
\label{b1:involution-solutions}
The following hold.
\begin{enumerate}
\item\label{b1:involution-cardinality}
The set \(\mathcal I_q\) has \(q^2\) elements.
\item\label{b1:involution-root-count}
For each
\(M\in\mathcal I_q\), the equation
\begin{equation}
\label{b1:frobenius-involution}
 a^{q^n}=M(a)
\end{equation}
has exactly \(q^n+1\) solutions in \(\PP^1(k)\).
\item\label{b1:involution-disjointness}
The solution sets of \eqref{b1:frobenius-involution} for
distinct elements of \(\mathcal I_q\) are disjoint outside
\(\PP^1(\F_{q^2})\).
\end{enumerate}
\end{lemma}

\begin{proof}
Each element of \(\mathcal I_q\) can be written uniquely in one of the forms
\[
 M(T)=-T+\beta\quad(\beta\in\F_q),\qquad
 M(T)=\frac{\alpha T+\beta}{T-\alpha}
 \quad(\alpha,\beta\in\F_q,\ \alpha^2+\beta\ne0).
\]
Hence \(|\mathcal I_q|=q+q(q-1)=q^2\).

In the first case, the finite solutions of
\eqref{b1:frobenius-involution} satisfy \(a^{q^n}+a=\beta\).
The trace map of \(k/\F_{q^n}\) gives \(q^n\) solutions in \(k\),
and \(\infty\) is also a solution.
In the second case, \eqref{b1:frobenius-involution} is equivalent to
\((a-\alpha)^{q^n+1}=\alpha^2+\beta\).
Since the right-hand side is nonzero, the norm map of
\(k/\F_{q^n}\) gives \(q^n+1\) solutions in \(k\).

If \(M(a)=M'(a)\) for distinct projective transformations over
\(\F_q\), then \(a\) satisfies a nonzero polynomial over \(\F_q\)
of degree at most two, or \(a=\infty\).  Thus
\(a\in\PP^1(\F_{q^2})\).
\end{proof}

\Needspace{8\baselineskip}
\begin{lemma}
\label{b1:ramified-rational-places}
The total number of rational places of \(F\) lying above the places
\(Q_a\), \(a\in\PP^1(\F_{q^2})\), is \(q^3+1\) when \(n\) is
odd, and \(q^2+1\) when \(n\) is even.
\end{lemma}

\begin{proof}
For even \(n\), these places of \(k(u)\) are totally ramified in
\(F/k(u)\), by Lemma~\ref{b1:ramification}(\ref{b1:ramification-indices}), and give \(q^2+1\)
rational places of \(F\).

Suppose now that \(n\) is odd. The places \(Q_a\) with
\(a\in\PP^1(\F_q)\) give \(q+1\) rational places of \(F\).
Since \(r=(q+1)\ell_0\), the function
\(w=z^m/B(u)^{\ell_0}\) satisfies
\[
 w^{q+1}=cA(u).
\]
The function \(cA(u)\) has a simple zero at \(u=0\), so
\(T^{q+1}-cA(u)\) is irreducible and
\([k(u,w):k(u)]=q+1\).

For \(a\in\F_{q^2}\setminus\F_q\), we have
\(A(a)^{q^n}=-A(a)\) and \(c^{q^n}=-c\), hence
\(cA(a)\in\F_{q^n}^*\).  Since \((q+1)\mid q^n+1\), every element of
\(\F_{q^n}^*\) is a \((q+1)\)-st power in \(k\). Thus
\(T^{q+1}-cA(a)\) has \(q+1\) distinct roots in \(k\).
By \cite[Corollary~3.3.8(c)]{Stichtenoth2009}, the place \(Q_a\)
splits completely in \(k(u,w)\).

Now \([F:k(u,w)]=m\), and
Lemma~\ref{b1:ramification}(\ref{b1:ramification-indices})
gives ramification index \(m\) over \(k(u)\) at every place of
\(F\) above \(Q_a\). Multiplicativity of ramification indices
therefore shows that each of the \(q+1\) rational places of
\(k(u,w)\) above \(Q_a\) has a unique extension to \(F\), which
is rational. Hence \(Q_a\) has \(q+1\) rational extensions to \(F\),
giving
\(q+1+(q^2-q)(q+1)=q^3+1\) rational places in total.
\end{proof}

\begin{theorem}
\label{b1:maximality}
For every prime power \(q\) and every integer \(n\ge1\), the curve
\(\mathcal X_{q,n}\) is maximal over \(\F_{q^{2n}}\).
\end{theorem}

\begin{proof}
For \(n=1\), the assertion follows from Remark~\ref{b1:small-members}.
Assume henceforth that \(n\ge2\).
We first consider the solutions of \eqref{b1:frobenius-involution}
in \(k\setminus\F_{q^2}\). Choose a trace-zero representative
\(M=\left(\begin{smallmatrix}\alpha&\beta\\\gamma&\delta\end{smallmatrix}\right)\)
and let \(a\in k\setminus\F_{q^2}\) satisfy this equation.
Using \(a^{q^n}=M(a)\) and \(\alpha+\delta=0\), we obtain
\[
 (\gamma a+\delta)^{q^n+1}
 =(\gamma M(a)+\delta)(\gamma a+\delta)
 =\gamma\beta+\delta^2=-\det(M).
\]
Lemma~\ref{b1:mobius-identities} now yields
\[
 f(a)^{q^n}=-f(a^{q^n})=-f(M(a))=f(a).
\]
Thus \(f(a)\in\F_{q^n}^*\), so the place \(Q_a\) splits completely
in \(F/k(u)\).

Suppose first that \(n\) is odd.  On \(\F_{q^2}\), the
\(q^n\)-power map is the \(q\)-power map. The trace and norm calculations
in Lemma~\ref{b1:involution-solutions}, applied to \(\F_{q^2}/\F_q\),
show that \(a^q=M(a)\) has exactly \(q+1\) solutions in
\(\PP^1(\F_{q^2})\). Therefore each element of \(\mathcal I_q\)
gives \(q^n-q\) rational places \(Q_a\), with \(a\notin\F_{q^2}\),
that split completely in \(F/k(u)\). The sets of places obtained
from distinct elements are disjoint by
Lemma~\ref{b1:involution-solutions}(\ref{b1:involution-disjointness}), giving
\(q^2(q^n-q)\) such places.

Suppose next that \(n\) is even.  The solutions in
\(\PP^1(\F_{q^2})\) are the fixed points of \(M\).
If \(q\) is odd, every involution has two distinct fixed points in
\(\PP^1(\F_{q^2})\).  Thus each of the \(q^2\) involutions contributes
\(q^n-1\) places.  If \(q\) is even, each nonidentity involution has
one fixed point, which belongs to \(\PP^1(\F_q)\), and contributes
\(q^n\) places. The identity contributes \(q^n-q^2\) places, giving
\((q^2-1)q^n+(q^n-q^2)=q^2(q^n-1)\) places in total.

Together with Lemma~\ref{b1:ramified-rational-places}, these
completely split places give
\begin{equation}
\label{b1:rational-count}
 \#\mathcal X_{q,n}(k)\ \ge\
 \begin{cases}
 (q^n+1)q^2(q^n-q)+q^3+1,&n\text{ odd},\\
 (q^n+1)q^2(q^n-1)+q^2+1,&n\text{ even}.
 \end{cases}
\end{equation}
By \eqref{b1:genus-formula}, each right-hand side equals
\(q^{2n}+1+2g(F)q^n\). The Hasse--Weil bound forces
equality, proving maximality.
\end{proof}

Equality in \eqref{b1:rational-count} shows that the places counted
in the proof are all the rational places of \(F\).
In particular, for \(a\in k\setminus\F_{q^2}\), the place \(Q_a\)
splits completely in \(F/k(u)\) if and only if
\eqref{b1:frobenius-involution} holds for some
\(M\in\mathcal I_q\); such an \(M\) is unique.

\subsection{Relation with known families}
\label{b1:known-members-subsection}

We now identify the members with odd \(n\ge3\) with the
Beelen--Montanucci curves.

\begin{proposition}
\label{b1:bm-identification}
For odd \(n\ge3\), the curve \(\mathcal X_{q,n}\) is
\(\F_{q^{2n}}\)-isomorphic to the Beelen--Montanucci curve.
\end{proposition}

\begin{proof}
Let \(k(x,y,w)\) be the function field of the Beelen--Montanucci
curve given by \eqref{b1:bm-model}.
Choose \(\theta\in\F_{q^2}\setminus\F_q\). Since \(n\) is odd,
\(c=(\theta-\theta^q)^{-1}\) satisfies \(c^{q^n}=-c\).
By the rescaling argument in Section~\ref{b1:ramification-subsection},
we may use this choice.

The second equation in \eqref{b1:bm-model} can be written as
\(w^m=xy^{q^2}-yx^{q^2}\), since
\[
 y^q(xy^{q^2}-yx^{q^2})
 =x(x^{q+1}-1)^q-x^{q^2}(x^{q+1}-1)
 =x^{q^2}-x.
\]
Put \(s=x+y\) and \(t=\theta x+\theta^qy\). Then
\begin{equation}
\label{b1:bm-st}
 \begin{aligned}
 c(s^qt-st^q)&=x^{q+1}-y^{q+1}=1,\\
 c(s^{q^2}t-st^{q^2})&=xy^{q^2}-yx^{q^2}=w^m.
 \end{aligned}
\end{equation}
Setting \(u=s/t\), equation~\eqref{b1:bm-st} gives
\begin{equation}
\label{b1:bm-ab}
 A(u)=\frac{1}{ct^{q+1}},\qquad
 B(u)=\frac{u^{q^2}-u}{u^q-u}=\frac{w^m}{t^{q^2-q}}.
\end{equation}

Since \(m r=\ell_0(q^n+1)\) and
\(q+1+(q^2-q)r=q^n+1\), equation~\eqref{b1:bm-ab} gives
\[
 cA(u)B(u)^r
 =\frac{w^{m r}}{t^{q+1+(q^2-q)r}}
 =\left(\frac{w^{\ell_0}}{t}\right)^{q^n+1}.
\]
Thus \eqref{b1:model} is satisfied by
\begin{equation}
\label{b1:bm-change}
 u=\frac{x+y}{\theta x+\theta^qy},\qquad
 z=\frac{w^{\ell_0}}{\theta x+\theta^qy}.
\end{equation}

\Needspace{10\baselineskip}
To recover the original generators, use
\(m-(q^2-q)\ell_0=1\) and \eqref{b1:bm-ab} to obtain
\[
 w=\frac{B(u)}{z^{q^2-q}},\qquad t=\frac{w^{\ell_0}}{z}.
\]
Then \(s=ut\), and solving for \(x,y\) gives
\[
 x=ct(1-\theta^qu),\qquad y=ct(\theta u-1).
\]
Thus \(k(x,y,w)=k(u,z)\), proving the assertion.
\end{proof}

For \(n=3\), the curve \(\mathcal X_{q,3}\) is \(k\)-isomorphic to
the Giulietti--Korchm\'aros curve; see
Proposition~\ref{b1:bm-identification} and \cite[Section~2]{BM2018}.

For odd \(n\ge3\), \eqref{b1:genus-formula} agrees with the computation in
\cite[Proposition~2.1]{BM2018}.

\begin{remark}
\label{b1:bm-maximality}
For \(n=3\), Theorem~\ref{b1:maximality} gives an independent
proof of maximality for the GK curve \cite[Theorem~1]{GK2009}.
For odd \(n\ge5\), it gives an independent proof for the BM
curves \cite[Theorem~3.9]{BM2018}.
\end{remark}

\section{Non-covering by the Hermitian curve}
\label{b1:noncover-section}

We prove that \(\mathcal X_{q,n}\) and the GGS curves are not
covered by \(\mathcal H_{q^n}\) over \(k\), for \(q>2\) and
\(n\ge3\) (with \(n\) odd in the GGS case).
The members of \(\mathcal X_{q,n}\) with odd \(n\) and the GGS curves
satisfy the common criterion below; the members with even \(n\)
require a comparison of factor multiplicities. We retain \(N=q^n+1\).

\subsection{Polynomial representations and a non-covering criterion}

The first lemma represents functions whose poles are bounded by
\(ND\) as ratios of homogeneous polynomials of degree \(\deg D\).

\Needspace{8\baselineskip}
\begin{lemma}
\label{b1:polynomial-representation}
Let \(D\) be an effective \(k\)-rational divisor of degree
\(1\le d<q^n\) on \(\mathcal H_{q^n}:b^{q^n}+b=a^N\).
Let \(1,f_1,\ldots,f_s\) be nonzero functions in
\(\mathcal L(ND)\), computed over \(\overline{k}\).
\begin{enumerate}
\item\label{b1:low-degree-nonvanishing}
No nonzero homogeneous polynomial of degree less than \(N\)
vanishes identically on \(\mathcal H_{q^n}\).
\item\label{b1:homogeneous-representation}
There are homogeneous polynomials
\(F_0,\ldots,F_s\in\overline{k}[T_1,T_2,T_3]\) of degree \(d\)
such that
\(f_i=F_i(a,b,1)/F_0(a,b,1)\).
\item\label{b1:coprime-representation}
If \((f_1)_\infty=ND\), the polynomials in
\textup{(\ref{b1:homogeneous-representation})} can be chosen with
\(\gcd(F_0,F_1)=1\).
\end{enumerate}
\end{lemma}

\begin{proof}
Let \(O\) be the unique place at infinity of \(\mathcal H_{q^n}\),
where \(a,b\) have pole orders \(q^n,N\), respectively.
For \textup{(\ref{b1:low-degree-nonvanishing})}, let \(G\ne0\) be
homogeneous with \(\deg G<N\). The monomials \(a^i b^t\) in
\(G(a,b,1)\) have pairwise distinct pole orders \(iq^n+tN\)
at \(O\), since \(\gcd(q^n,N)=1\) and
\(0\le i\le\deg G<N\). Hence \(G(a,b,1)\ne0\).

To prove \textup{(\ref{b1:homogeneous-representation})}, recall that
the divisor relation of Fuhrmann, Garcia, and Torres
\cite[Corollary~1.2]{FGT1997} gives
\begin{equation}
\label{b1:frobenius-divisor-relation}
 q^nP+\operatorname{Fr}(P)\sim NO
 \qquad(P\in\mathcal H_{q^n}(\overline{k})),
\end{equation}
where \(\operatorname{Fr}\) is the \(q^{2n}\)-power Frobenius.

Since \(D\) is \(k\)-rational, \(\operatorname{Fr}(D)=D\).
Summing \eqref{b1:frobenius-divisor-relation} over \(D\) yields
\[
 ND=q^nD+\operatorname{Fr}(D)\sim dNO.
\]
Choose a function \(h\) with \((h)=ND-dNO\). Multiplication by
\(h\) identifies \(\mathcal L(ND)\) with \(\mathcal L(dNO)\), so
\(h,hf_1,\ldots,hf_s\in\mathcal L(dNO)\).

The Weierstrass semigroup at \(O\) is \(\langle q^n,N\rangle\):
these are the pole orders of \(a,b\), and the semigroup they generate
has \(q^n(q^n-1)/2\) gaps, equal to the genus of \(\mathcal H_{q^n}\).
Every nongap at \(O\) has a unique expression \(iq^n+tN\) with
\(0\le i\le q^n\) and \(t\ge0\). Thus \(\mathcal L(dNO)\) has a basis consisting of the
monomials \(a^i b^t\) with \(0\le i\le q^n\), \(t\ge0\), and
\(iq^n+tN\le dN\): their pole orders are pairwise distinct and
exhaust the nongaps not exceeding \(dN\).
Since \(d<q^n\), the inequality \(i+t>d\) would imply
\(iq^n+tN\ge(d+1)q^n>dN\). Homogenizing the resulting expressions
for \(h,hf_1,\ldots,hf_s\) therefore gives the required polynomials.
This proves \textup{(\ref{b1:homogeneous-representation})}.

To prove \textup{(\ref{b1:coprime-representation})}, put
\(E_G=(G(a,b,1))+N(\deg G)O\) for each nonzero homogeneous
polynomial \(G\) with \(\deg G<N\).
Each monomial has pole order at most \(N\deg G\), so \(E_G\) is
effective of degree \(N\deg G\).
These divisors satisfy \(E_{GH}=E_G+E_H\)
whenever \(\deg(GH)<N\).
If \((f_1)_\infty=ND\), the choice of \(h\) and the homogenization give
\[
 \begin{aligned}
 E_{F_0}&=(h)+dNO=ND=(f_1)_\infty,\\
 E_{F_1}&=(hf_1)+dNO=(f_1)_0.
 \end{aligned}
\]
The divisors \(E_{F_0}\) and \(E_{F_1}\) have disjoint supports.
A common factor \(G\) of positive degree would give a nonzero effective divisor
\(E_G\) contained in both, a contradiction.
\end{proof}

The next lemma applies when two functions with poles bounded by
\(NP\) satisfy a Hermitian equation.

\begin{lemma}
\label{b1:hermitian-relation-criterion}
Let \(X\) be a nonsingular, geometrically irreducible projective
curve over \(k\). Suppose that \(P\in X(k)\) and that
\(x,y\in\mathcal L(NP)\), with \(x\) nonconstant, satisfy
\[
 x^q+x=y^{q+1}.
\]
If \(2g(X)-2>(q+1)(q^n-2)\), then there is no nonconstant
\(k\)-morphism \(\mathcal H_{q^n}\to X\).
\end{lemma}

\begin{proof}
We first reduce to separable coverings. If a covering exists,
view \(k(X)\) as a subfield of \(L_0=k(a,b)\) through its pullback,
and let \(E\) be the maximal intermediate field separable over
\(k(X)\). Then \(L_0/E\) is purely inseparable of degree \(p^s\)
for some \(s\ge0\). Since \(k\) is perfect,
\cite[Proposition~3.10.2(c)]{Stichtenoth2009} gives
\(E=L_0^{p^s}=k(a^{p^s},b^{p^s})\).
The equation \(b^{q^n}+b=a^N\) has coefficients in \(\F_p\),
so \(a\mapsto a^{p^s}\), \(b\mapsto b^{p^s}\), fixing \(k\),
defines a \(k\)-isomorphism \(L_0\to E\).
Composing \(k(X)\hookrightarrow E\) with its inverse yields a
separable covering \(\pi:\mathcal H_{q^n}\to X\).

Let \(d\) be its degree. The Hurwitz genus formula gives
\[
 d(2g(X)-2)\le q^n(q^n-1)-2=N(q^n-2).
\]
Together with the genus hypothesis, this implies \(q^n>2\) and
\((q+1)d<N\), hence \(d<q^n\).
Applying Lemma~\ref{b1:polynomial-representation}(\ref{b1:homogeneous-representation})
to \(1,\pi^*x,\pi^*y\in\mathcal L(N\pi^*P)\) gives
homogeneous polynomials \(F_0,F_1,F_2\) of degree \(d\) over
\(\overline{k}\), with restrictions
\(\pi^*x=F_1/F_0\) and \(\pi^*y=F_2/F_0\).
The polynomial \(F_1^qF_0+F_1F_0^q-F_2^{q+1}\) vanishes on
\(\mathcal H_{q^n}\) and has degree \((q+1)d<N\). By
Lemma~\ref{b1:polynomial-representation}(\ref{b1:low-degree-nonvanishing}),
\begin{equation}
\label{b1:hermitian-polynomial-identity}
 F_1^qF_0+F_1F_0^q=F_2^{q+1}
\end{equation}
as an identity of homogeneous polynomials.
Thus \((F_0:F_1:F_2)\) defines a rational map
\(\PP^2\dashrightarrow\mathcal H_q\) over \(\overline{k}\),
nonconstant because \(\pi^*x\) is nonconstant.

Choose two \(\overline{k}\)-rational points in its domain with
distinct images. Restricting to the line through them gives a
nonconstant rational map \(\PP^1\dashrightarrow\mathcal H_q\),
hence an embedding
\(\overline{k}(\mathcal H_q)\hookrightarrow\overline{k}(\PP^1)\).
L\"uroth's theorem gives \(g(\mathcal H_q)=0\), contrary to
\(g(\mathcal H_q)=q(q-1)/2>0\).
\end{proof}

\subsection{The family \texorpdfstring{\(\mathcal X_{q,n}\)}{X(q,n)}}

The even case uses the following sequence of functions in
\(\mathcal L(NP_\infty)\).

For even \(n\ge4\), we have
\(\ell_0=(q^{n-1}-q)/(q^2-1)\ge q\) and
\(r=(q+1)\ell_0+1\). Put
\begin{equation}
\label{b1:ell}
 j=q^2-q,\qquad v=\frac{z^j}{B(u)}.
\end{equation}

\begin{lemma}
\label{b1:pole-functions}
For every prime power \(q\) and every even \(n\ge4\), the function \(v\)
is nonconstant, and
\begin{equation}
\label{b1:function-list}
 1,u,z,zv,\ldots,zv^{\ell_0}\in\mathcal L(NP_\infty).
\end{equation}
\end{lemma}

\begin{proof}
By Lemma~\ref{b1:ramification}(\ref{b1:infinite-place}), the unique pole of \(u\) is
\(P_\infty\), and
\(\operatorname{ord}_{P_\infty}(u)=-N\),
\(\operatorname{ord}_{P_\infty}(z)=-q^n\).  Hence
\(\operatorname{ord}_{P_\infty}(v)=j>0\), so \(v\) is nonconstant.
For \(0\le i\le\ell_0\),
\(\operatorname{ord}_{P_\infty}(zv^i)=-q^n+ij\ge-N\).
For every place \(P\) of \(F\)
above \(Q_a\), with \(a\in\F_{q^2}\setminus\F_q\),
\[
 \operatorname{ord}_{P}(zv^i)
 =r-i(q+1)\ge1
 \qquad(0\le i\le\ell_0),
\]
since these places are totally ramified and \(r=(q+1)\ell_0+1\).
For \(P\mid Q_a\) with \(a\in\F_q\),
\(\operatorname{ord}_{P}(zv^i)=1+ij>0\).
The functions \(zv^i\) have no other finite poles, proving
\eqref{b1:function-list}.
\end{proof}

\Needspace{6\baselineskip}
\begin{theorem}
\label{b1:noncover}
Let \(q>2\) and \(n\ge3\). There is no nonconstant
\(\F_{q^{2n}}\)-morphism
\(\mathcal H_{q^n}\to\mathcal X_{q,n}\).
\end{theorem}

\begin{proof}
Suppose first that \(n\) is odd. By the constant rescaling in
Section~\ref{b1:model-section}, we may choose
\(c\in\F_{q^2}^*\) with \(c^q=-c\).
Take \(w=z^m/B(u)^{\ell_0}\) as in
Lemma~\ref{b1:ramified-rational-places}, and put \(x=-cu\),
\(y=w\). The equation \(w^{q+1}=cA(u)\) gives
\begin{equation}
\label{b1:odd-hermitian-functions}
 x^q+x=y^{q+1},\qquad
 (x)_\infty=NP_\infty,\qquad
 (y)_\infty=qm P_\infty.
\end{equation}
Indeed, the equation for \(w\) shows that it has no finite
poles and has pole order \(qN/(q+1)=qm\) at \(P_\infty\).
Thus \(x,y\in\mathcal L(NP_\infty)\), with \(x\) nonconstant.
The subfield \(k(u,w)=k(x,y)\) is Hermitian, and
\(F/k(u,w)\) is cyclic of degree \(m\); under
Proposition~\ref{b1:bm-identification}, this is the Hermitian
covering in the BM construction.
By \eqref{b1:genus-formula},
\[
 2g(\mathcal X_{q,n})-2-(q+1)(q^n-2)
 =(q+1)(q-2)(q^n-q)>0.
\]
Lemma~\ref{b1:hermitian-relation-criterion} proves the assertion.

Now let \(n\ge4\) be even. The reduction in the proof of
Lemma~\ref{b1:hermitian-relation-criterion} allows us to assume
that a covering \(\pi:\mathcal H_{q^n}\to\mathcal X_{q,n}\)
is separable, of degree \(d\).
We will use \((j+1)d<N\) to turn the pullback relations into
polynomial identities, and \(d<j\ell_0+1\) for the final
comparison of factor multiplicities.

The Hurwitz genus formula gives
\begin{equation}
\label{b1:hurwitz-bound}
 \begin{aligned}
 d(2g(\mathcal X_{q,n})-2)
 &\le 2g(\mathcal H_{q^n})-2\\
 &=q^n(q^n-1)-2=N(q^n-2).
 \end{aligned}
\end{equation}
Equation \eqref{b1:genus-formula} gives
\[
 2g(\mathcal X_{q,n})-2-(j+1)(q^n-2)
 =(q-2)q^n+2j>0.
\]
Together with \eqref{b1:hurwitz-bound}, this yields
\begin{equation}
\label{b1:polynomial-degree-bound}
 (j+1)d<N.
\end{equation}
In particular, \(d<q^n\).
Since \(N=jr+q+1\), \(r\ge q+1\), and
\(r=(q+1)\ell_0+1\), we obtain
\begin{equation}
\label{b1:cover-degree}
 d<\frac{N}{j+1}\le r=(q+1)\ell_0+1<j\ell_0+1,
\end{equation}
where the last inequality uses \(q>2\), hence \(j>q+1\), and
\(\ell_0\ge1\).

Put \(D=\pi^*P_\infty\). This is an effective \(k\)-rational
divisor of degree \(d\). The pullbacks of \eqref{b1:function-list}
belong to \(\mathcal L(ND)\), and the pullback of \(u\) has pole
divisor \(ND\). By Lemma~\ref{b1:polynomial-representation}(\ref{b1:homogeneous-representation})--(\ref{b1:coprime-representation}), they
are represented by nonzero homogeneous polynomials
\(X_0,X_1,Y_0,\ldots,Y_{\ell_0}\) of degree \(d\) over \(\overline{k}\),
with \(\gcd(X_0,X_1)=1\). On \(\mathcal H_{q^n}\), we have
\(u=X_1/X_0\) and \(zv^i=Y_i/X_0\), suppressing the pullback
notation. Hence their restrictions satisfy
\begin{equation}
\label{b1:quadratic-identities}
 Y_0Y_{i+1}=Y_1Y_i,\qquad 0\le i<\ell_0.
\end{equation}
Since
\(B(u)=\prod_{\alpha\in\F_{q^2}\setminus\F_q}(u-\alpha)\),
the relation \(vB(u)=z^j\) also gives
\begin{equation}
\label{b1:factor-identity}
 Y_1\prod_{\alpha\in\F_{q^2}\setminus\F_q}
 (X_1-\alpha X_0)=Y_0^{j+1}.
\end{equation}
The homogeneous equations \eqref{b1:quadratic-identities} and
\eqref{b1:factor-identity} have degrees \(2d\) and \((j+1)d\),
respectively. Both are less than \(N\) by
\eqref{b1:polynomial-degree-bound}. Hence
Lemma~\ref{b1:polynomial-representation}(\ref{b1:low-degree-nonvanishing}) shows that these equations
are polynomial identities.

The restriction of \(Y_1/Y_0\) is the pullback of the nonconstant
function \(v\). As \(Y_0,Y_1\) have the same degree,
\(Y_0\) does not divide \(Y_1\).
Write \(\operatorname{ord}_\varpi(H)\) for the multiplicity of an
irreducible polynomial \(\varpi\) in a nonzero polynomial \(H\).
There is therefore an irreducible \(\varpi\) with
\(\operatorname{ord}_\varpi(Y_0)>\operatorname{ord}_\varpi(Y_1)\).
Put \(a_0=\operatorname{ord}_\varpi(Y_0)\) and
\(e=a_0-\operatorname{ord}_\varpi(Y_1)>0\).
Equation~\eqref{b1:quadratic-identities} gives
\(\operatorname{ord}_\varpi(Y_i)=a_0-ie\) for \(0\le i\le\ell_0\),
so \(a_0\ge\ell_0 e\).
The factors \(X_1-\alpha X_0\) are pairwise relatively prime:
a nonconstant common factor for distinct \(\alpha,\alpha'\) would divide
\((\alpha'-\alpha)X_0\), hence both \(X_0\) and \(X_1\),
contrary to \(\gcd(X_0,X_1)=1\). Thus \(\varpi\) divides at most one of
them, each of degree \(d\). Equations~\eqref{b1:factor-identity}
and \eqref{b1:cover-degree} give
\[
 d\ge\operatorname{ord}_\varpi\!\left(
 \prod_{\alpha\in\F_{q^2}\setminus\F_q}(X_1-\alpha X_0)\right)
 =ja_0+e\ge(j\ell_0+1)e>d,
\]
a contradiction.
\end{proof}

\Needspace{10\baselineskip}
\subsection{The GGS curves}

For the GGS family, Galois non-covering was proved for \(q>2\)
in \cite{DM2012} and for \(q=2\), \(n\ge5\), in \cite{GMZ2016}.
The question of arbitrary coverings is recorded in
\cite[Section~2]{TTT2016}. The following application of
Lemma~\ref{b1:hermitian-relation-criterion} excludes such coverings
for \(q>2\).

\begin{theorem}
\label{b1:ggs-noncover}
Let \(q>2\) and let \(n\ge3\) be odd. The GGS curve \(\mathcal C\)
over \(\F_{q^{2n}}\), with affine equations
\begin{equation}
\label{b1:ggs-model}
 x^q+x=y^{q+1},\qquad w^m=y^{q^2}-y,
 \qquad m=\frac{q^n+1}{q+1},
\end{equation}
is not covered by \(\mathcal H_{q^n}\) over \(\F_{q^{2n}}\).
\end{theorem}

\begin{proof}
On the Hermitian function field \(k(x,y)\), the functions \(x,y\)
have a unique pole, of orders \(q+1,q\), respectively. The pole
order of \(y^{q^2}-y\) is \(q^3\), relatively prime to \(m\).
Hence this place is totally ramified in
\(k(x,y,w)/k(x,y)\), and its unique extension \(P_\infty\) is
\(k\)-rational. We have
\((x)_\infty=NP_\infty\) and
\((y)_\infty=qm P_\infty\), so
\(x,y\in\mathcal L(NP_\infty)\).
The genus formula in \cite[Proposition~2.2]{GGS2010} gives
\[
 2g(\mathcal C)-2-(q+1)(q^n-2)
 =(q+1)(q-2)(q^n-q)>0.
\]
Lemma~\ref{b1:hermitian-relation-criterion} now applies.
\end{proof}

\subsection{The case \texorpdfstring{\(q=2\)}{q=2}}

\begin{proposition}
\label{b1:q2-cover}
For \(q=2\), the curve \(\mathcal X_{2,n}\) is a cyclic Galois
subcover of \(\mathcal H_{2^n}\) for every \(n\ge1\).
\end{proposition}

\begin{proof}
For \(n=1\), this follows from Remark~\ref{b1:small-members}.
For odd \(n\ge3\), a cyclic Galois covering of degree \((2^n+1)/3\)
follows from Proposition~\ref{b1:bm-identification} and the proof of
\cite[Lemma~2.4]{BM2018}.

Suppose that \(n\) is even. By Section~\ref{b1:model-section},
we may take \(c=1\). Choose \(\omega\in\F_4\setminus\F_2\).
Put \(d=(2^n-1)/3\) and use
the model \(y^{2^n+1}=x^{2^n}+x\) for \(\mathcal H_{2^n}\).
A covering of degree \(d\) is given by
\begin{equation}
\label{b1:q2-even-map}
 u=\frac{x^d+1}{\omega x^d+\omega^2},\qquad
 z=\frac{yx^{(d-1)/2}}{\omega x^d+\omega^2}.
\end{equation}
Put \(t=x^d\) and \(D=\omega t+\omega^2\). Since
\(3d=2^n-1\) and \(r=2^{n-1}-1\), we have
\[
 1+\frac{(d-1)(2^n+1)}2=dr,\qquad 2r+3=2^n+1.
\]
Thus \(y^{2^n+1}=x(x^{3d}+1)\) gives
\(z^{2^n+1}=t^r(t^3+1)/D^{2^n+1}\). Moreover,
\[
 A(u)=\frac{(t+1)(\omega^2t+\omega)}{D^2},\qquad
 B(u)=\frac{t}{D^2},\qquad
 (\omega^2t+\omega)D=t^2+t+1.
\]
Using \((t+1)(t^2+t+1)=t^3+1\), we obtain
\[
 A(u)B(u)^r
 =\frac{t^r(t+1)(\omega^2t+\omega)}{D^{2r+2}}
 =\frac{t^r(t^3+1)}{D^{2^n+1}}=z^{2^n+1}.
\]

Since \(k(u)=k(t)\) and \(k(x,y)=k(u,z)(x)\), the relation
\(x^d=t\) shows that the covering has degree at most \(d\).
As \(d\mid2^n-1\), the maps
\((x,y)\mapsto(\zeta^2x,\zeta y)\), with
\(\zeta\in k\) and \(\zeta^d=1\), form a cyclic group of order
\(d\) fixing \(u,z\). Hence the covering is cyclic Galois of
degree \(d\).
For \(n=2\), formula~\eqref{b1:q2-even-map} gives an isomorphism.
\end{proof}

\section{The automorphism group}
\label{b1:automorphisms-section}

Throughout this section, assume that \(n\ge4\).
For both odd and even \(n\), equation~\eqref{b1:mobius-f} allows us
to extend fractional linear transformations of \(u\) over \(\F_q\)
to automorphisms of \(F\).
By Theorem~\ref{b1:maximality} and \cite[Theorem~3.10]{GSY2015},
all automorphisms of \(\mathcal X_{q,n}\) are defined over \(k\).
For odd \(n\ge5\), the full automorphism group was determined in
\cite[Theorem~4.3]{BM2018}.

\begin{proposition}
\label{b1:automorphism-subgroup}
Assume that \(n\ge4\). For
\(M=\left(\begin{smallmatrix}\alpha&\beta\\\gamma&\delta\end{smallmatrix}\right)
\in\GL_2(q)\) and \(\lambda\in k^*\) with \(\lambda^{q^n+1}=\det M\),
the assignments
\begin{equation}
\label{b1:automorphism-action}
 \sigma_{M,\lambda}(u)=\frac{\alpha u+\beta}{\gamma u+\delta},\qquad
 \sigma_{M,\lambda}(z)=\frac{\lambda z}{\gamma u+\delta}
\end{equation}
define \(k\)-automorphisms of \(F=k(u,z)\). These automorphisms
form a subgroup of \(\Aut(\mathcal X_{q,n})\) of order
\((q^n+1)q(q^2-1)\).
\end{proposition}

\begin{proof}
Equations~\eqref{b1:model} and \eqref{b1:mobius-f} give
\[
 \left(\frac{\lambda z}{\gamma u+\delta}\right)^{q^n+1}
 =\frac{\det(M)f(u)}{(\gamma u+\delta)^{q^n+1}}
 =f\!\left(\frac{\alpha u+\beta}{\gamma u+\delta}\right).
\]
Thus \eqref{b1:automorphism-action} preserves \eqref{b1:model}.
Composition gives
\begin{equation}
\label{b1:automorphism-composition}
 \sigma_{M,\lambda}\circ\sigma_{M',\lambda'}
 =\sigma_{M'M,\lambda'\lambda}.
\end{equation}
The matrix order is reversed because the maps act on functions
by substitution. Their inverses are
\(\sigma_{M^{-1},\lambda^{-1}}\), so they form a subgroup.

The map \(\lambda\mapsto\lambda^{q^n+1}\) is the norm from \(k\) to
\(\F_{q^n}\), so each \(M\) admits exactly \(q^n+1\) choices of
\(\lambda\). If two pairs define the same automorphism, comparison
of their images of \(u\) gives \(M'=tM\) for some \(t\in\F_q^*\).
Comparison of their images of \(z\) then gives
\(\lambda'=t\lambda\). Conversely, \((tM,t\lambda)\) defines the
same automorphism, and
\((t\lambda)^{q^n+1}=t^2\det M=\det(tM)\).
Thus each automorphism is represented by exactly \(q-1\) pairs,
and the subgroup has order
\((q^n+1)|\GL_2(q)|/(q-1)=(q^n+1)q(q^2-1)\).
\end{proof}

Put \(\mathcal A=\Aut(\mathcal X_{q,n})\), and let
\(G\le\mathcal A\) be the subgroup in
Proposition~\ref{b1:automorphism-subgroup}.
Choose a primitive \(N\)-th root of unity \(\zeta\in k\), and let
\(\rho\in G\) be given by \(\rho(u)=u\), \(\rho(z)=\zeta z\).
Then \(C=\langle\rho\rangle\) has order \(N\), \(F^C=k(u)\),
and every element of \(G\) commutes with \(\rho\).

\Needspace{6\baselineskip}
\begin{lemma}
\label{b1:commuting-automorphisms}
Suppose that \(n\ge4\) is even. The centralizer of \(\rho\)
in \(\mathcal A\) is \(G\).
\end{lemma}

\begin{proof}
Let \(\sigma\in\mathcal A\) commute with \(\rho\).
Both \(\sigma\) and \(\sigma^{-1}\) preserve \(F^C=k(u)\), so
\(\sigma(u)=M(u)\) for some \(M\in\PGL_2(k)\).
Moreover,
\[
 \rho\!\left(\frac{\sigma(z)}{z}\right)
 =\frac{\sigma(\rho(z))}{\rho(z)}
 =\frac{\sigma(z)}{z}.
\]
Hence \(\sigma(z)=a(u)z\) with \(a(u)\in k(u)^*\), and
\begin{equation}
\label{b1:centralizer-identity}
 f(M(u))=a(u)^Nf(u).
\end{equation}
By \eqref{b1:kummer-valuations}, the places of \(k(u)\)
where \(f(u)\) has valuation congruent to \(1\) modulo \(N\)
are exactly those corresponding to \(\PP^1(\F_q)\), since
\(-q^n=1-N\) and \(1<r<N\).
For a geometric place \(Q\) of \(k(u)\), let \(M(Q)\) denote
its image under the fractional linear transformation on \(\PP^1\).
Taking valuations in \eqref{b1:centralizer-identity} gives
\[
 v_{M(Q)}(f)=v_Q(f(M(u)))\equiv v_Q(f)\pmod N.
\]
Hence \(M(\PP^1(\F_q))=\PP^1(\F_q)\).
Thus \(M\in\PGL_2(q)\), since its images of
\(0,1,\infty\) determine a fractional linear transformation over
\(\F_q\). Choose a lift of \(M\) in \(G\) using
Proposition~\ref{b1:automorphism-subgroup}. Composing \(\sigma\)
with the inverse of this lift gives an automorphism fixing \(u\),
which belongs to \(C\). Hence \(\sigma\in G\).
\end{proof}

By maximality and \cite[Proposition~3.12]{GSY2015}, every
nontrivial \(p\)-subgroup of \(\mathcal A\) fixes a unique place of
\(\overline{k}(u,z)/\overline{k}\) and acts freely on its other places.
We use this property to compare the \(p\)-parts of the
point stabilizers in \(G\) and \(\mathcal A\).

\begin{theorem}
\label{b1:full-automorphism-group}
For every prime power \(q\) and every integer \(n\ge4\), the
subgroup in Proposition~\ref{b1:automorphism-subgroup} is the full
automorphism group of \(\mathcal X_{q,n}\). In particular,
\(\lvert\Aut(\mathcal X_{q,n})\rvert=(q^n+1)q(q^2-1)\).
\end{theorem}

\begin{proof}
For odd \(n\ge5\), the assertion follows from
Proposition~\ref{b1:bm-identification} and
\cite[Theorem~4.3]{BM2018}. Assume henceforth that \(n\ge4\) is
even. In the rest of the proof, we extend constants to \(\overline{k}\)
in all function fields and retain the same notation. All places are
considered over \(\overline{k}\).
We first determine the ramified \(G\)-orbits and their
stabilizers, and then prove that \(\mathcal A_P=G_P\) at
each place in these orbits. Comparing two applications of the
Hurwitz genus formula will then give \(\mathcal A=G\).

Let \(\Omega_1\) and \(\Omega_2\) be the sets of places of \(F\)
lying above \(\PP^1(\F_q)\) and
\(\F_{q^2}\setminus\F_q\), respectively, and put
\(\Sigma=\Omega_1\cup\Omega_2\).
By Lemma~\ref{b1:ramification}, each of the corresponding places
of \(k(u)\) is totally ramified in \(F\). Hence
\(\lvert\Omega_1\rvert=q+1\),
\(\lvert\Omega_2\rvert=q^2-q\), and \(C\) fixes \(\Sigma\)
pointwise. Proposition~\ref{b1:genus} gives
\(2g(F)=(q^2-1)(N-1)\).

The action of \(G\) on \(k(u)\) induces \(\PGL_2(q)\), which
is transitive on each of the two sets of base places. For the
second set, the affine transformations already act transitively:
if \(a,b\in\F_{q^2}\setminus\F_q\), then
\(b=sa+t\) for some \(s\in\F_q^*\), \(t\in\F_q\).
It follows that \(\Omega_1,\Omega_2\) are \(G\)-orbits, with
\begin{equation}
\label{b1:known-stabilizers}
 |G_P|=
 \begin{cases}
 Nq(q-1),&P\in\Omega_1,\\
 N(q+1),&P\in\Omega_2.
 \end{cases}
\end{equation}
Every nonidentity fractional linear transformation over \(\F_q\)
has its fixed places in \(\PP^1(\F_{q^2})\), since its fixed-point
equation has degree at most two. A nonidentity element of \(C\)
fixes no place outside \(\Sigma\), since \(F/k(u)\) is
unramified there. Thus \(G_P=1\) for \(P\notin\Sigma\).
Also, \(F^G\subseteq k(u)\), so \(F^G\) is rational by
L\"uroth's theorem.

We prove that \(\mathcal A_P=G_P\) for every \(P\in\Sigma\).
Fix such a place and put
\(U=\mathcal A_P^{(1)}\), \(U_0=U\cap G\), and \(s=|U_0|\).
By \cite[Proposition~3.8.5(e)]{Stichtenoth2009},
\(U\) is a normal \(p\)-subgroup of \(\mathcal A_P\), and
\(\mathcal A_P/U\) is cyclic of order prime to \(p\).
The definition of the lower ramification groups gives
\(G_P^{(i)}=G_P\cap\mathcal A_P^{(i)}\) for \(i\ge0\).
Thus \(U_0=G_P^{(1)}\), and \eqref{b1:known-stabilizers}
gives \(s=q\) on \(\Omega_1\) and \(s=1\) on \(\Omega_2\).
Since \(U\) fixes \(P\), it acts freely on the other places,
trivially for \(U=1\) and by \cite[Proposition~3.12]{GSY2015}
for \(U\ne1\).

We first prove that \(U=U_0\). Put \(E=F^U\).
Since \(C\le\mathcal A_P\) and
\(U\trianglelefteq\mathcal A_P\), the group \(C\) acts on \(E\).
This action is faithful: an element of \(C\) fixing \(E\) belongs
to \(\operatorname{Gal}(F/E)=U\), and \(C\cap U=1\).
Suppose that \(R,R'\in\Sigma\setminus\{P\}\) restrict to the
same place of \(E\). Then \(R'=\sigma_0(R)\) for some
\(\sigma_0\in U\). As \(\rho\) fixes both places,
\[
 (\rho\sigma_0\rho^{-1})(R)=R'=\sigma_0(R).
\]
Thus \(\sigma_0^{-1}\rho\sigma_0\rho^{-1}\in U\) fixes \(R\ne P\).
The free action of \(U\) away from \(P\) gives
\(\rho\sigma_0\rho^{-1}=\sigma_0\), and
Lemma~\ref{b1:commuting-automorphisms} yields \(\sigma_0\in U_0\).
Thus at most \(s\) places of \(\Sigma\setminus\{P\}\) restrict
to any one place of \(E\). The place below \(P\) receives no
other place of \(F\), since \(U\) fixes \(P\).
Therefore \(E\) has at least \(1+q^2/s\) distinct places fixed
by \(C\).

These places are totally ramified in the tame extension
\(E/E^C\), of degree \(N\) by the faithful action of \(C\).
Using \(g(E^C)\ge0\), the Hurwitz genus formula
\cite[Theorem~3.4.13]{Stichtenoth2009} gives
\begin{equation}
\label{b1:fixed-field-genus-bound}
 2g(E)-2\ge-2N+\left(1+\frac{q^2}{s}\right)(N-1),
 \qquad
 2g(E)\ge\left(\frac{q^2}{s}-1\right)(N-1).
\end{equation}
At \(P\), the zeroth and first ramification groups of \(F/E\)
are both \(U\). Hilbert's different formula
\cite[Theorem~3.8.7]{Stichtenoth2009} gives a different exponent
at least \(2(|U|-1)\). Applying the Hurwitz genus formula to
\(F/E\), we obtain
\begin{equation}
\label{b1:wild-genus-bound}
 2g(F)-2\ge |U|(2g(E)-2)+2(|U|-1),
 \qquad 2g(F)\ge |U|\,2g(E).
\end{equation}
Since \(2g(F)=(q^2-1)(N-1)\),
\eqref{b1:fixed-field-genus-bound} and \eqref{b1:wild-genus-bound} imply
\[
 |U|\le\frac{q^2-1}{q^2/s-1}
 =\begin{cases}
 q+1,&P\in\Omega_1,\\
 1,&P\in\Omega_2.
 \end{cases}
\]
For \(P\in\Omega_1\), \(q=|U_0|\) divides \(|U|\), so
\(|U|=q\); for \(P\in\Omega_2\), \(|U|=1\).
Hence \(U=U_0\) in both cases.

For \(\tau\in\mathcal A_P\), the cyclicity of
\(\mathcal A_P/U\) gives
\(\tau\rho\tau^{-1}=\sigma_1\rho\) for some \(\sigma_1\in U\).
Since \(U=U_0\le G\), the elements \(\sigma_1\) and \(\rho\)
commute. Taking \(N\)-th powers gives
\(1=(\sigma_1\rho)^N=\sigma_1^N\).
The order of \(\sigma_1\) is a power of \(p\), whereas \(p\nmid N\),
so \(\sigma_1=1\). Lemma~\ref{b1:commuting-automorphisms} now gives
\(\tau\in G\). We have proved
\begin{equation}
\label{b1:equal-stabilizers}
 \mathcal A_P=G_P\qquad(P\in\Sigma).
\end{equation}

Finally, put \(h=[\mathcal A:G]\) and choose \(P_i\in\Omega_i\)
for \(i=1,2\). The orbits \(\mathcal A\cdot P_1\) and
\(\mathcal A\cdot P_2\) are distinct, since the \(p\)-parts of
their stabilizer orders are \(q\) and \(1\).
By \eqref{b1:equal-stabilizers},
\(|\mathcal A\cdot P_i|=h|\Omega_i|\) for \(i=1,2\).
Let \(d_i\) be the different exponent at \(P_i\) in \(F/F^G\).
By \eqref{b1:equal-stabilizers},
\(G_{P_i}^{(j)}=G_{P_i}\cap\mathcal A_{P_i}^{(j)}
=\mathcal A_{P_i}^{(j)}\) for every \(j\ge0\).
Hilbert's different formula therefore gives the different exponent
\(d_i\) in \(F/F^{\mathcal A}\) as well.
Since \(F^G\) is rational and \(G\) has no ramification outside
\(\Sigma\), the Hurwitz genus formula gives
\[
 2g(F)-2=-2|G|+(q+1)d_1+(q^2-q)d_2.
\]
\Needspace{6\baselineskip}
Applying the Hurwitz genus formula to \(F/F^{\mathcal A}\) and
retaining only the contributions of \(\mathcal A\cdot P_1\) and
\(\mathcal A\cdot P_2\) gives
\[
 2g(F)-2\ge-2|\mathcal A|+h(q+1)d_1+h(q^2-q)d_2
 =h(2g(F)-2).
\]
As \(g(F)>1\), this forces \(h=1\). Hence \(\mathcal A=G\).
\end{proof}

Henceforth \(G=\Aut(\mathcal X_{q,n})\).

\Needspace{11\baselineskip}
\begin{proposition}
\label{b1:group-structure}
Suppose that \(n\ge4\) is even. The center of \(G\) is \(C\),
and \(G/C\cong\PGL_2(q)\).
\begin{enumerate}
\item If \(q\) is even, then \(G\cong C_N\times\SL_2(q)\).
\item If \(q\) is odd, put
\begin{equation}
\label{b1:double-cover-matrix-group}
 K=\{a^{-1}M:M\in\GL_2(q),\ a\in\F_{q^2}^*,\ a^2=\det M\}
   \le\SL_2(q^2).
\end{equation}
Then \(G\cong C_{N/2}\times K\), where
\(K/\{I,-I\}\cong\PGL_2(q)\). The unique involution of \(G\)
is \(\rho^{N/2}\).
\end{enumerate}
The two direct factors have coprime orders in both cases.
\end{proposition}

\begin{proof}
The map \(\sigma_{M,\lambda}\mapsto[M^{-1}]\) is a surjective
homomorphism \(G\to\PGL_2(q)\) with kernel \(C\). Since \(C\) is central
and \(\PGL_2(q)\) has trivial center, the center of \(G\) is
\(C\).

For \(a\in\F_{q^2}\), the assumption that \(n\) is even gives
\(a^N=a^2\). Thus the pairs \((M,a)\) with \(a^2=\det M\)
in Proposition~\ref{b1:automorphism-subgroup} define a subgroup
of \(G\). When \(q\) is even, the square root \(a\) belongs
to \(\F_q\) and is unique. The map
\(T\mapsto\sigma_{T^{-1},1}\) is an isomorphism from
\(\SL_2(q)\) onto this subgroup, which intersects
\(C\) trivially and maps onto \(\PGL_2(q)\).
Hence \(G\cong C_N\times\SL_2(q)\).

When \(q\) is odd, the same subgroup is isomorphic to \(K\).
More precisely, \(a^{-1}M\mapsto\sigma_{M^{-1},a^{-1}}\)
is an isomorphism. Its intersection with \(C\)
has order \(2\). Since \(q^n\equiv1\pmod8\), the integer
\(N/2\) is odd. The subgroup of \(C\) of order \(N/2\)
therefore intersects this copy of \(K\) trivially. Their product
has order \(Nq(q^2-1)=|G|\), proving the asserted decomposition.

For even \(n\),
\begin{equation}
\label{b1:central-coprimality}
 \gcd\bigl(N,q(q^2-1)\bigr)=\gcd(N,2).
\end{equation}
Indeed, \(q^n\equiv1\pmod{q^2-1}\). This proves the claims
about coprime orders. In odd characteristic, the only
nonidentity element of order \(2\) in \(\SL_2(q^2)\)
is \(-I\): an involution is diagonalizable, and its determinant
is \(1\). The decomposition of \(G\) therefore shows that
\(\rho^{N/2}\) is the unique involution of \(G\).
\end{proof}

Since \(F^C=k(u)\) and \(G/C\cong\PGL_2(q)\), we have
\(F^G=k(u)^{\PGL_2(q)}\). The following corollary gives a
generator of this rational field, together with the ramification
and orbit data of the full group.

\Needspace{10\baselineskip}
\begin{corollary}
\label{b1:full-quotient-orbits}
Suppose that \(n\ge4\) is even. Let \(\Omega_1\) and \(\Omega_2\)
be the sets of points of \(\mathcal X_{q,n}(\overline{k})\) corresponding
to the geometric places above \(\PP^1(\F_q)\) and
\(\F_{q^2}\setminus\F_q\), respectively.
\begin{enumerate}
\item The fixed field is
\begin{equation}
\label{b1:full-group-quotient}
 F^G=k(t_0),\qquad t_0=\frac{B(u)^{q+1}}{A(u)^{q^2-q}}.
\end{equation}
The only places of \(k(t_0)\) ramified in \(F\) are the pole
of \(t_0\), below \(\Omega_1\), and the zero of \(t_0\), below
\(\Omega_2\). They are wildly and tamely ramified, respectively.
\item The only short \(G\)-orbits are \(\Omega_1\) and
\(\Omega_2\), of lengths \(q+1\) and \(q^2-q\), with stabilizer
orders \(Nq(q-1)\) and \(N(q+1)\), respectively.
Both consist of \(k\)-rational points. The remaining \(k\)-rational
points form \(q(q^n-1)/(q^2-1)\) orbits of length \(|G|\).
\end{enumerate}
\end{corollary}

\begin{proof}
Equation~\eqref{b1:mobius-ab} gives
\(\sigma_{M,\lambda}(t_0)=(\det M)^{-(q^2-q)}t_0=t_0\)
for every \(\sigma_{M,\lambda}\in G\). Since \(A\) and \(B\) are coprime,
\[
 [k(u):k(t_0)]=(q+1)(q^2-q)=|\PGL_2(q)|.
\]
Thus \(k(u)^{\PGL_2(q)}=k(t_0)\), and \(F^C=k(u)\) gives
\eqref{b1:full-group-quotient}.

The two short orbits and their stabilizers follow from the proof of
Theorem~\ref{b1:full-automorphism-group} and \eqref{b1:known-stabilizers}.
In \(k(u)\), the zeros of \(t_0\) are the zeros of \(B(u)\),
and its poles correspond to \(\PP^1(\F_q)\).
This proves the ramification assertions.
By Lemma~\ref{b1:ramified-rational-places}, all \(q^2+1\)
points in \(\Omega_1\cup\Omega_2\) are \(k\)-rational.
By maximality and \eqref{b1:genus-formula},
\(\#\mathcal X_{q,n}(k)=q^{2n+2}+1\).
Since \(G\) acts over \(k\) and freely outside \(\Omega_1\cup\Omega_2\),
the remaining \(k\)-rational points form
\[
 \frac{q^{2n+2}+1-(q^2+1)}{Nq(q^2-1)}
 =\frac{q(q^n-1)}{q^2-1}
\]
orbits of length \(|G|\).
\end{proof}

\section{All Galois subfields for even \texorpdfstring{\(n\)}{n}}
\label{b1:quotient-equations-section}

Throughout Sections~\ref{b1:quotient-equations-section}
and~\ref{b1:explicit-quotient-families}, \(n\ge4\) is even and
\(G=\Aut(\mathcal X_{q,n})\). We parameterize all subgroups
of \(G\) and determine equations and genera for their fixed fields.

We retain \(N=q^n+1\), \(f(T)=cA(T)B(T)^r\), and the central subgroup
\(C\) of Section~\ref{b1:automorphisms-section}.
Let \(\pi:G\longrightarrow\PGL_2(q)\) be the homomorphism
given by \(\pi(\sigma_{M,\lambda})=[M^{-1}]\), as follows from
the composition law \eqref{b1:automorphism-composition}.
All quotient curves in these sections are \(k\)-maximal, by
Theorem~\ref{b1:maximality} and Serre's theorem recalled in
Section~\ref{b1:introduction}.

For odd \(n\), \(\gcd(q^n+1,q(q^2-1))=q+1\), whereas
\eqref{b1:central-coprimality} leaves only a possible common
factor \(2\) for even \(n\). For every admissible pair
\((d,S)\) in Proposition~\ref{b1:subgroup-realization}, this gives
\(\gcd(N/d,|S|)=1\). This coprimality gives subgroup uniqueness
and the explicit generators in Theorem~\ref{b1:even-quotient-equations}.

\subsection{Subgroup parameters}
For \(U\le G\), put
\begin{equation}
\label{b1:subgroup-parameters}
 D=U\cap C,\qquad d=|D|,\qquad \mu=N/d,\qquad
 S=\pi(U),\qquad \ell=|S|.
\end{equation}
Thus \(|U|=d\ell\).

\begin{proposition}
\label{b1:subgroup-realization}
Let \(d\mid N\), and let \(S\le\PGL_2(q)\) be a given
subgroup. There is a subgroup \(U\le G\) with the parameters
\((d,S)\) in \eqref{b1:subgroup-parameters} if and only if
\(q\) is even, \(d\) is even, or \(|S|\) is odd.
Whenever it exists, \(U\) is unique, and
\(\gcd(\mu,\ell)=1\).
\end{proposition}

\begin{proof}
If \(q,d\) are odd and \(\ell\) is even, then \(U\) would
have even order. By Cauchy's theorem, it would contain the
unique involution of \(G\), which belongs to \(C\) by
Proposition~\ref{b1:group-structure}. This
contradicts \(|U\cap C|=d\).

For the remaining parameters, equation~\eqref{b1:central-coprimality} gives
\(\gcd(\mu,\ell)=1\) when \(q\) is even. For odd \(q\),
the integer \(N\) is divisible by \(2\) but not by \(4\).
Either \(d\) is even and \(\mu\) is odd, or \(\ell\) is
odd. Equation~\eqref{b1:central-coprimality} therefore gives
\(\gcd(\mu,\ell)=1\).

Let \(D\) be the unique subgroup of \(C\) of order \(d\),
and put \(H=\pi^{-1}(S)\). The subgroup \(C/D\) is central
in \(H/D\), with order \(\mu\) and index \(\ell\).
The Schur--Zassenhaus theorem gives a complement; since \(C/D\)
is cyclic, any two complements are conjugate by an element of
\(C/D\). Centrality makes this complement unique.
Its inverse image in \(H\) is the unique subgroup \(U\)
with parameters \((d,S)\).
\end{proof}

\subsection{A Kummer equation for every subgroup}

Let \(U\le G\), with the notation in
\eqref{b1:subgroup-parameters}.
Since \(C\) is central in \(G\), we have
\(U\trianglelefteq CU\) and \(F^{CU}=k(u)^S\).
Thus \(F^U/k(u)^S\) is cyclic of degree \(\mu\), with
Galois group isomorphic to \(C/D\).

We choose a free \(S\)-orbit of places of \(k(u)\) unramified
in \(F\). This orbit will give
both the poles of a generator of \(k(u)^S\) and the
denominator of the Kummer generator.
Since \(n\) is even, \(\F_{q^4}\subseteq k\).
Choose \(\xi\in\F_{q^4}\setminus\F_{q^2}\), and put
\begin{equation}
\label{b1:quotient-rational-generator}
 t=t_S(u)=\sum_{s\in S}\frac{1}{s(u)-\xi},\qquad
 P_S(T)=\prod_{s\in S}(T-s(\xi)).
\end{equation}
We evaluate the summands on \(\PP^1\), with
\(1/(\infty-\xi)=0\). Define
\begin{equation}
\label{b1:quotient-orbit-polynomials}
 \Phi_S(T)=\prod_{a\in\PP^1(\F_q)}(T-t_S(a)),\qquad
 \Psi_S(T)=\prod_{b\in\F_{q^2}\setminus\F_q}(T-t_S(b)),
\end{equation}
and \(\kappa_S=f(\xi)^\ell/P_S'(\xi)^N\).

\begin{theorem}
\label{b1:even-quotient-equations}
With the notation in
\eqref{b1:subgroup-parameters}--\eqref{b1:quotient-orbit-polynomials},
put \(v=z^{d\ell}/P_S(u)^d\). Then
\begin{equation}
\label{b1:even-quotient-model}
 F^U=k(t,v),\qquad
 v^\mu=\kappa_S\Phi_S(t)\Psi_S(t)^r.
\end{equation}
Here \(P_S,\Phi_S,\Psi_S\in\F_{q^4}[T]\) and \(\kappa_S\in k^*\).
\end{theorem}

\begin{proof}
A nonidentity element of \(\PGL_2(q)\) fixes only points of
\(\PP^1(\F_{q^2})\). The stabilizer of \(\xi\) in \(S\)
is therefore trivial. Thus \(P_S\in\F_{q^4}[T]\) has
\(\ell\) distinct roots and \(P_S'(\xi)\ne0\).
The summands defining \(t_S(u)\) in
\eqref{b1:quotient-rational-generator} have
\(\ell\) distinct simple poles, each occurring in only one summand,
so none can cancel. Their sum is \(S\)-invariant. Consequently,
\([k(u):k(t)]=\ell\), so \(k(u)^S=k(t)\). All values in
\eqref{b1:quotient-orbit-polynomials} are finite and lie in
\(\F_{q^4}\). Since \(k(t)=k(u)^S\), the function \(t_S\)
defines the quotient map, whose geometric fibers are precisely
the \(S\)-orbits.
Thus \(\Phi_S,\Psi_S\in\F_{q^4}[T]\), of degrees
\(q+1,q^2-q\), respectively, and each distinct root has
multiplicity equal to the size of its corresponding \(S\)-orbit.
Since \(f(\xi)\ne0\), also \(\kappa_S\in k^*\).

We have \(F^D=k(u,z^d)\). By \eqref{b1:automorphism-action},
the conjugates of \(z^d/(u-\xi)^d\) under \(U/D\) are
nonzero constant multiples of \(z^d/(u-s(\xi))^d\),
\(s\in S\). Taking their product gives
\[
 \operatorname{Norm}_{F^D/F^U}
 \left(\frac{z^d}{(u-\xi)^d}\right)
 =\eta_0\,\frac{z^{d\ell}}{P_S(u)^d}=\eta_0 v,
 \qquad \eta_0\in k^*.
\]
Hence \(v\in F^U\). Since \(vP_S(u)^d=(z^d)^\ell\),
\((z^d)^\mu=f(u)\), and \(\gcd(\mu,\ell)=1\),
B\'ezout's identity gives \(k(u,v)=k(u,z^d)=F^D\). Thus
\[
 [F:k(t,v)]\le d[k(u):k(t)]=d\ell=|U|.
\]
As \(k(t,v)\subseteq F^U\), equality holds, and \(F^U=k(t,v)\).

By the definition of \(v\) and \(d\mu=N\), we have
\(v^\mu=f(u)^\ell/P_S(u)^N\).
Let \(a\) be an indeterminate over \(k(u)\).
Since \(t_S(u)\) has denominator \(P_S(u)\) and is regular at
infinity, \(P_S(u)(t_S(u)-t_S(a))\) is a polynomial in
\(k(a)[u]\) of degree at most \(\ell\). It vanishes at the
\(\ell\) distinct elements \(s(a)\), \(s\in S\), by the
\(S\)-invariance of \(t_S\). Since \(s(\xi)\ne\xi\)
for \(s\ne1\), all summands in \(t_S(u)\) except
\(1/(u-\xi)\) are regular at \(u=\xi\). Therefore
\[
 \left.P_S(u)\bigl(t_S(u)-t_S(a)\bigr)\right|_{u=\xi}
 =\left.\frac{P_S(u)}{u-\xi}\right|_{u=\xi}
 =P_S'(\xi).
\]
Hence
\begin{equation}
\label{b1:quotient-difference-factorization}
 P_S(u)\bigl(t_S(u)-t_S(a)\bigr)
 =P_S'(\xi)\prod_{s\in S}\frac{u-s(a)}{\xi-s(a)}.
\end{equation}
After cancelling the denominators of \(s(a)\) in each ratio,
\eqref{b1:quotient-difference-factorization} specializes to every
\(a\in\PP^1(\F_{q^2})\); a factor with \(s(a)=\infty\) has value \(1\).

Each \(s\in S\) permutes both \(\PP^1(\F_q)\) and
\(\F_{q^2}\setminus\F_q\). Taking products of
\eqref{b1:quotient-difference-factorization} over these two sets gives
\begin{equation}
\label{b1:quotient-orbit-factorization}
 \begin{aligned}
 \Phi_S(t)&=\left(\frac{P_S'(\xi)}{P_S(u)}\right)^{q+1}
             \left(\frac{A(u)}{A(\xi)}\right)^\ell,\\
 \Psi_S(t)&=\left(\frac{P_S'(\xi)}{P_S(u)}\right)^{q^2-q}
             \left(\frac{B(u)}{B(\xi)}\right)^\ell.
 \end{aligned}
\end{equation}
Since \(q+1+r(q^2-q)=N\) and \(f=cAB^r\),
\eqref{b1:quotient-orbit-factorization} yields
\[
 \Phi_S(t)\Psi_S(t)^r
 =\frac{P_S'(\xi)^N}{f(\xi)^\ell}
  \frac{f(u)^\ell}{P_S(u)^N}
 =\frac{v^\mu}{\kappa_S}.
\]
This proves \eqref{b1:even-quotient-model}.
\end{proof}

For fixed \(S\) and \(\xi\), the polynomials
\(P_S,\Phi_S,\Psi_S\) are independent of \(n\).

Write \(r=a_0\mu+\bar r\), where \(0\le\bar r<\mu\).
Replacing \(v\) in \eqref{b1:even-quotient-model} by
\(v/\Psi_S(t)^{a_0}\) gives
\begin{equation}
\label{b1:reduced-quotient-model}
 v^\mu=\kappa_S\Phi_S(t)\Psi_S(t)^{\bar r}.
\end{equation}
The polynomial on the right has degree at most
\(q+1+(\mu-1)(q^2-q)\).

\subsection{The genera of all quotient curves}
\label{b1:all-quotient-genera-subsection}

Recall that \(\nu(S)\) counts the elements of order \(p\) in
\(S\le\PGL_2(q)\).

\begin{theorem}
\label{b1:all-even-quotient-genera}
Let \(U\le G\), with the parameters \(d,\mu,S,\ell\) in
\eqref{b1:subgroup-parameters}. Then
\begin{equation}
\label{b1:all-even-quotient-genus-formula}
 g(F^U)=\frac{q^2-1-\nu(S)}{2\ell}(\mu-1).
\end{equation}
\end{theorem}

\begin{proof}
If \(\mu=1\), then \(F^U=k(t)\), and the formula follows.
Assume henceforth that \(\mu>1\).
Equation~\eqref{b1:even-quotient-model} has the form
\(v^\mu=\kappa_S\Phi_S(t)\Psi_S(t)^r\).
The distinct roots of \(\Phi_S\Psi_S\) correspond to the
\(S\)-orbits on \(\PP^1(\F_{q^2})\).
In \(\Phi_S\Psi_S^r\), the root corresponding to an orbit
in \(\PP^1(\F_q)\) has multiplicity equal to the orbit size;
for an orbit in \(\F_{q^2}\setminus\F_q\), its multiplicity
is \(r\) times the orbit size.
Each orbit size divides \(\ell\).
Proposition~\ref{b1:subgroup-realization} gives \(\gcd(\mu,\ell)=1\).
Also \(\mu\mid N\) and \(\gcd(N,r)=1\) by
Section~\ref{b1:ramification-subsection}, so \(\gcd(\mu,r)=1\).
Thus all these places are totally ramified in \(F^U/k(t)\). The degree of \(\Phi_S\Psi_S^r\) is
\(q+1+r(q^2-q)=N=d\mu\), so the infinite place is unramified.
Let \(b(S)\) be the number of these orbits. Since \(p\nmid \mu\),
each ramified place has different exponent \(\mu-1\). The Hurwitz genus
formula \cite[Theorem~3.4.13]{Stichtenoth2009} gives
\[
 2g(F^U)-2=-2\mu+b(S)(\mu-1),
\]
hence \(2g(F^U)=(b(S)-2)(\mu-1)\).

A nonidentity element of \(\PGL_2(q)\) fixes one point of
\(\PP^1(\F_{q^2})\) if it has order \(p\), and two otherwise.
Burnside's formula therefore gives
\[
 b(S)=\frac{q^2+1+\nu(S)+2(\ell-1-\nu(S))}{\ell}
     =2+\frac{q^2-1-\nu(S)}{\ell}.\qedhere
\]
\end{proof}

Write \(q=p^a\). Table~\ref{b1:genus-parameter-table} lists the
subgroup types and the parameters \((|S|,\nu(S))\); the classification
and existence conditions follow from \cite[Theorem~D]{Faber2023}.
Here \(D_{2s}\) denotes a dihedral group of order \(2s\), and
\(V\rtimes C_s\) acts by \(u\mapsto\lambda u+\beta\), where
\(\beta\in V\le(\F_q,+)\), \(|V|=p^b\), and
\(\lambda\in C_s\le\F_q^*\) preserves \(V\) under multiplication.
The full set of Galois quotient genera is obtained from
\eqref{b1:all-even-quotient-genus-formula} by taking every row and
every \(d\mid N\) allowed by Proposition~\ref{b1:subgroup-realization}.
Subgroups with the same pair \((|S|,\nu(S))\) give the same
genus for each admissible \(d\).
In small characteristics, some groups in rows (vii)--(ix) also occur
in rows (iv)--(vi); see \cite[Remark~2.3]{Faber2023}.

\begin{table}[!htbp]
\centering
\caption{Parameters for all Galois quotient genera when \(n\ge4\) is even.}
\label{b1:genus-parameter-table}
\small
\setlength{\tabcolsep}{4pt}
\renewcommand{\arraystretch}{1.35}
\begin{tabular}{r c c c c}
\hline
Row&\(S\)&Existence conditions&\(|S|\)&\(\nu(S)\)\\
\hline
(i)&\(1\)&Always&\(1\)&\(0\)\\
(ii)&\(C_s\)&\(s\ge2,\ s\mid q-1\) or \(s\mid q+1\)&\(s\)&\(0\)\\
(iii)&\(D_{2s}\)&\(s\ge2,\ s\mid q-1\) or \(s\mid q+1\)&\(2s\)&
\(\begin{cases}0,&p>2,\\s,&p=2\end{cases}\)\\[3pt]
(iv)&\(V\rtimes C_s\)&\(1\le b\le a,\ s\mid\gcd(p^b-1,q-1)\)&\(p^b s\)&\(p^b-1\)\\
(v)&\(\PSL_2(p^b)\)&\(b\mid a\)&
\(\dfrac{p^b(p^{2b}-1)}{\gcd(2,p^b-1)}\)&\(p^{2b}-1\)\\
(vi)&\(\PGL_2(p^b)\)&\(b\mid a\)&\(p^b(p^{2b}-1)\)&\(p^{2b}-1\)\\
(vii)&\(A_4\)&\(p>2\), or \(p=2,\ 2\mid a\)&\(12\)&
\(\begin{cases}3,&p=2,\\8,&p=3,\\0,&p>3\end{cases}\)\\[5pt]
(viii)&\(S_4\)&\(p>2\)&\(24\)&
\(\begin{cases}8,&p=3,\\0,&p>3\end{cases}\)\\[5pt]
(ix)&\(A_5\)&\(q\equiv0,\pm1\pmod5\)&\(60\)&
\(\begin{cases}15,&p=2,\\20,&p=3,\\24,&p=5,\\0,&p>5\end{cases}\)\\
\hline
\end{tabular}
\end{table}
\FloatBarrier

We verify the entries in Table~\ref{b1:genus-parameter-table}.
In rows (ii)--(iii), \(p\nmid s\), so \(\nu(C_s)=0\).
For \(D_{2s}\), if \(p=2\), precisely the \(s\) reflections have
order \(p\); if \(p>2\), then \(p\nmid2s\), so \(\nu(D_{2s})=0\).

For row (iv) with \(s>1\), let \(\lambda\) be a primitive
\(s\)-th root of unity and let \(e\) be the multiplicative order
of \(p\) modulo \(s\). Since \(\F_p(\lambda)=\F_{p^e}\),
we have \(\lambda\in\F_q\) if and only if \(e\mid a\), and
\(C_s\)-invariant additive subgroups are precisely
\(\F_{p^e}\)-subspaces. Thus a suitable \(V\subseteq\F_q\)
of order \(p^b\) exists if and only if \(e\mid\gcd(a,b)\),
equivalently \(s\mid\gcd(p^b-1,q-1)\); take
\(\dim_{\F_{p^e}}V=b/e\).
For \(s=1\), any \(b\)-dimensional \(\F_p\)-subspace suffices.
Since \(u\mapsto\lambda u+\beta\) with \(\lambda\ne1\) is
conjugate by a translation to \(u\mapsto\lambda u\), its order
is prime to \(p\). The elements of order \(p\) are therefore
exactly the nonidentity translations, giving
\(\nu(V\rtimes C_s)=p^b-1\).

For rows (v)--(vi), the translation group of order \(p^b\)
is the unique Sylow \(p\)-subgroup of the stabilizer of \(\infty\).
By transitivity on \(\PP^1(\F_{p^b})\), its conjugates are
indexed by the \(p^b+1\) points. Every nonidentity element
of such a subgroup has order \(p\) and fixes only the corresponding
point. Thus these subgroups intersect trivially, giving
\(\nu(S)=(p^b+1)(p^b-1)=p^{2b}-1\).

The values of \(\nu(S)\) in rows (vii)--(ix) follow by counting
double transpositions, three-cycles, and five-cycles.

\Needspace{12\baselineskip}
\section{Closed-form quotient models}
\label{b1:explicit-quotient-families}

For the subgroup types below, we eliminate the orbit polynomials
\(\Phi_S,\Psi_S\) from the equations of
Theorem~\ref{b1:even-quotient-equations}.
All remaining subgroup types are covered by that theorem.
We retain the parameters \(d,\mu,S\) from
\eqref{b1:subgroup-parameters} and the invariant \(t_0\) from
\eqref{b1:full-group-quotient}. All parameters are subject to
Proposition~\ref{b1:subgroup-realization}, and all equations
describe their nonsingular projective models.

Table~\ref{b1:closed-form-guide} lists the models below; the displayed
functions generate \(k(u)^S\) in the coordinates used in each subsection.

\begin{table}[!htbp]
\centering
\caption{Closed-form quotient models.}
\label{b1:closed-form-guide}
\small
\setlength{\tabcolsep}{4pt}
\renewcommand{\arraystretch}{1.3}
\begin{tabular}{p{2.3in} c c c}
\hline
Subgroup type and conditions&Base generator&Equation&\(\mu\)\\
\hline
Affine \(V\rtimes C_s\), including split cyclic&
\(L_V(u)^s\)&\eqref{b1:affine-quotient-model}&Any\\
Nonsplit cyclic \(C_s\), \(s\mid q+1\)&
\(y=x^s\)&\eqref{b1:torus-cyclic-model}&Any\\
Dihedral \(D_{2s}\), \(s\mid q-\epsilon\), \(\epsilon=\pm1\)&
\(w=x_\epsilon^s+b/x_\epsilon^s\)&\eqref{b1:dihedral-dickson-model}&Odd\\
\(\PGL_2(Q)\), \(q=Q^2\)&
\(X\)&\eqref{b1:square-pgl-model}&Odd\\
\(\PSL_2(Q)\), \(q=Q^2\), \(Q\) odd&
\(W\)&\eqref{b1:square-psl-model}&Odd\\
\hline
\end{tabular}
\smallskip

\parbox{\linewidth}{All rows are subject to
Proposition~\ref{b1:subgroup-realization}; ``Any'' means every
admissible \(\mu\). For odd \(\mu\), the first two rows also admit
\eqref{b1:affine-uniform-model} and
\eqref{b1:torus-cyclic-odd-model}, respectively.}
\end{table}

Corollary~\ref{b1:conjugate-pair-quotients} gives further models for
the full tori and their normalizers, under the admissibility
conditions stated there.

We conjugate the projective image to the displayed form when
necessary and lift the conjugating element to \(\sigma\in G\).
Since \(F^{\sigma U\sigma^{-1}}=\sigma(F^U)\), applying
\(\sigma^{-1}\) to the displayed generators gives generators
for the original subgroup.

If \(\mu\) is even, Proposition~\ref{b1:subgroup-realization} and
Table~\ref{b1:genus-parameter-table} show that \(S\) is conjugate
to an affine or a nonsplit cyclic subgroup. The formulas in
Section~\ref{b1:point-subgroups} apply to every admissible \(\mu\)
for these types. The dihedral and projective subfield groups have
even order, so admissibility forces \(\mu\) to be odd. The following
proposition gives a common Kummer generator when \(\mu\) is odd.

\begin{proposition}
\label{b1:odd-kummer-generator}
Let \(U\le G\), with the parameters in
\eqref{b1:subgroup-parameters}, and suppose that \(\mu\) is odd.
Define
\begin{equation}
\label{b1:odd-kummer-generator-formula}
 v_0=(-c^d)^{q(q-1)/2}\frac{B(u)^d}{z^{dq(q-1)}}.
\end{equation}
For every \(t\) with \(k(u)^S=k(t)\), we have
\[
 F^U=k(t,v_0),\qquad v_0^\mu=t_0.
\]
If \(t\) is chosen with \(\F_q(u)^S=\F_q(t)\), then
\(t_0=R_{q,S}(t)\), where \(R_{q,S}\in\F_q(T)\) has degree
\([\PGL_2(q):S]\). For fixed \(q,S,t\), the function
\(R_{q,S}\) is independent of \(n\) and \(c\).
\end{proposition}

\begin{proof}
Since \(c^N=-c^2\) and \(\mu\) is odd,
\(\bigl((-c^d)^{q(q-1)/2}\bigr)^\mu=c^{q(q-1)}\).
Equation~\eqref{b1:model} and
\(N=q(q-1)r+q+1\) therefore give \(v_0^\mu=t_0\).
For \(\sigma_{M,\lambda}\in U\),
\eqref{b1:automorphism-action} and \eqref{b1:mobius-ab} give
\(\sigma_{M,\lambda}(v_0)/v_0=\lambda^{-dq(q-1)}\).
The function \(v_0\) belongs to \(F^D\), and \(v_0^\mu=t_0\in F^G\).
Thus \(\sigma D\mapsto\sigma(v_0)/v_0\) defines a homomorphism
from \(U/D\) to the group of \(\mu\)-th roots of unity in \(k\).
By Proposition~\ref{b1:subgroup-realization},
\(\gcd(\mu,|U/D|)=1\), so this homomorphism is trivial and
\(v_0\in F^U\).

Since \(\mu\) is odd, \eqref{b1:central-coprimality} gives
\(\gcd(\mu,q(q-1))=1\).
Together with \((z^d)^\mu=f(u)\),
\eqref{b1:odd-kummer-generator-formula} yields
\(k(u,v_0)=k(u,z^d)=F^D\).
Consequently,
\[
 [F:k(t,v_0)]\le d[k(u):k(t)]=d|S|=|U|.
\]
The inclusion \(k(t,v_0)\subseteq F^U\) forces equality.

By L\"uroth's theorem, we may choose \(t\) with
\(\F_q(u)^S=\F_q(t)\). Since \(t_0\in\F_q(u)^S\),
there is a rational function \(R_{q,S}\in\F_q(T)\) with
\(t_0=R_{q,S}(t)\). The degree computation in
Corollary~\ref{b1:full-quotient-orbits} gives
\[
 \deg R_{q,S}=[\F_q(t):\F_q(t_0)]=[\PGL_2(q):S].
\]
For fixed \(q,S,t\), neither \(t\) nor \(t_0\) depends on
\(n\) or \(c\), proving the last assertion.
\end{proof}

\subsection{Affine and cyclic subgroups}
\label{b1:point-subgroups}

Suppose first that \(S\) fixes a point of \(\PP^1(\F_q)\).
We may take \(S=V\rtimes C_s\), where \(V\le(\F_q,+)\)
and \(C_s\le\F_q^*\) preserves \(V\).
Put \(h=|V|\) and
\[
 L_V(X)=\prod_{a\in V}(X-a),\qquad
 W_V(T)=\prod_{b\in L_V(\F_q)}(T-b).
\]
The polynomials \(L_V,W_V\in\F_q[T]\) are additive;
see \cite[Lemma~3.1]{MW2026}.
Since \(L_V\) has kernel \(V\) on \(\F_q\), the composition
\(W_V(L_V(X))\) is monic of degree \(q\) and vanishes at every
element of \(\F_q\). Hence \(W_V(L_V(X))=X^q-X\).

\begin{proposition}
\label{b1:affine-quotient-equations}
There is a polynomial \(R_{V,s}\in\F_q[T]\) such that
\(W_V(Y)=Y R_{V,s}(Y^s)\).
The fixed field is \(F^U=k(t,v)\), where
\(t=L_V(u)^s\), \(v=z^{ds}\), and
\begin{equation}
\label{b1:affine-quotient-model}
 v^\mu=c^s t R_{V,s}(t)^s
 \bigl(1+t^{(q-1)/s}R_{V,s}(t)^{q-1}\bigr)^{rs}.
\end{equation}
\end{proposition}

\begin{proof}
The action of \(C_s\) on \(V\setminus\{0\}\) is free,
so \(h\equiv1\pmod s\). For \(a\in C_s\), this gives
\(L_V(aX)=a^hL_V(X)=aL_V(X)\). Applying
\(W_V(L_V(X))=X^q-X\) at \(aX\) gives
\[
 W_V(aL_V(X))=(aX)^q-aX=aW_V(L_V(X)),
\]
so \(W_V(aY)=aW_V(Y)\). Comparison of coefficients yields
\(W_V(Y)=Y R_{V,s}(Y^s)\).

The function \(t\) is fixed by \(U\). Put \(w=z^d\).
Since \(S\) fixes \(\infty\), the matrix \(M\) of every
\(\sigma_{M,\lambda}\in U\) has \(\gamma=0\).
Equation~\eqref{b1:automorphism-action} therefore gives
\(\sigma_{M,\lambda}(w)/w=(\lambda/\delta)^d\in k^*\).
As \(D\) fixes \(w\), the map
\(U/D\to k^*,\ \sigma D\mapsto\sigma(w)/w\)
is a well-defined homomorphism. Its image has order dividing
\(hs\) and prime to \(p\), hence dividing \(s\).
Thus \(\sigma(v)/v=(\sigma(w)/w)^s=1\) for every \(\sigma\in U\).
Since \(\gcd(\mu,s)=1\) and \(w^\mu=f(u)\), we have
\(k(u,v)=k(u,w)=F^D\). Hence
\([F:k(t,v)]\le d[k(u):k(t)]=dhs=|U|\), proving
\(F^U=k(t,v)\).
Substituting \(A(u)=L_V(u)R_{V,s}(t)\) into
\(v^\mu=f(u)^s\) gives \eqref{b1:affine-quotient-model}.
\end{proof}

\Needspace{11\baselineskip}
Since \(\nu(S)=h-1\) for \(S=V\rtimes C_s\),
Theorem~\ref{b1:all-even-quotient-genera} gives
\[
 g(F^U)=\frac{q^2-h}{2hs}\left(\frac{N}{d}-1\right).
\]
For \(s=1\), equation~\eqref{b1:affine-quotient-model} becomes
\begin{equation}
\label{b1:translation-quotient-model}
 t=L_V(u),\qquad v=z^d,\qquad
 v^\mu=cW_V(t)\bigl(1+W_V(t)^{q-1}\bigr)^r.
\end{equation}
Taking \(V=\{0\}\) gives \(t=u\) and
\(v^\mu=cA(u)B(u)^r\). Taking \(V=\F_q\) gives
\(t=u^q-u\), \(v^\mu=ct(1+t^{q-1})^r\), and genus
\((q-1)(N/d-1)/2\).

For odd \(\mu\), put
\[
 E_{V,s}(T)=T^{(q-1)/s}R_{V,s}(T)^{q-1}.
\]
Since \(E_{V,s}(t)=A(u)^{q-1}\),
Proposition~\ref{b1:odd-kummer-generator} gives
\begin{equation}
\label{b1:affine-uniform-model}
 F^U=k(t,v_0),\qquad
 v_0^\mu=\frac{(1+E_{V,s}(t))^{q+1}}{E_{V,s}(t)^q}.
\end{equation}
For the full affine group \(S=\F_q\rtimes\F_q^*\), take
\(t=A(u)^{q-1}\). Then
\begin{equation}
\label{b1:borel-model}
 v_0^\mu=\frac{(1+t)^{q+1}}{t^q},\qquad
 g(F^U)=\frac{\mu-1}{2}.
\end{equation}

We next consider cyclic subgroups of a nonsplit torus. Fix
\(\theta\in\F_{q^2}\setminus\F_q\), and put
\[
 x=\frac{u-\theta}{u-\theta^q},\qquad
 \Delta=\theta-\theta^q,\qquad
 H_q(T)=\sum_{i=0}^{q-2}T^i.
\]
Set \(z_1=(x-1)z\) and \(c'=c\Delta\). Since
\(\Delta\in\F_{q^n}\), we have \((c')^{q^n}=-c'\).
For \(s\mid q+1\), put \(e=(q+1)/s\) and
\(S=\{x\mapsto\zeta x:\zeta^s=1\}\).

\begin{proposition}
\label{b1:torus-cyclic}
Let \(U\) have admissible parameters \((d,S)\).
Then \(F^U=k(y,v)\), where \(y=x^s\), \(v=z_1^{ds}\), and
\begin{equation}
\label{b1:torus-cyclic-model}
 v^\mu=(c')^s y^r(y^e-1)^sH_q(y^e)^{rs}.
\end{equation}
\end{proposition}

\begin{proof}
Substitution of \(u=(\theta^q x-\theta)/(x-1)\) gives
\begin{equation}
\label{b1:torus-coordinate-factors}
 A(u)=\frac{\Delta(x^{q+1}-1)}{(x-1)^{q+1}},\qquad
 B(u)=\frac{xH_q(x^{q+1})}{(x-1)^{q(q-1)}}.
\end{equation}
Hence \(z_1^N=c'x^r(x^{q+1}-1)H_q(x^{q+1})^r\).
Since \(z_1=-\Delta z/(u-\theta^q)\) and \(S\) fixes
\(\theta^q\), formula~\eqref{b1:automorphism-action} gives
\(\sigma(z_1)/z_1\in k^*\) for every \(\sigma\in U\).
The multipliers of \(z_1^d\) define a character of \(U/D\),
whose order divides \(s\), so \(v\in F^U\).
Since \(\gcd(\mu,s)=1\), we have
\(k(x,v)=k(x,z_1^d)=F^D\), and
\([F:k(y,v)]\le ds=|U|\). This proves the fixed-field
assertion and \eqref{b1:torus-cyclic-model}.
\end{proof}

Theorem~\ref{b1:all-even-quotient-genera} gives
\(g(F^U)=(q^2-1)(\mu-1)/(2s)\).
For odd \(\mu\), equations~\eqref{b1:torus-coordinate-factors}
and \(\Delta^{q(q-1)}=-1\), together with
Proposition~\ref{b1:odd-kummer-generator}, give the alternative
\begin{equation}
\label{b1:torus-cyclic-odd-model}
 F^U=k(y,v),\qquad v=-v_0,\qquad
 v^\mu=\frac{y^eH_q(y^e)^{q+1}}{(y^e-1)^{q(q-1)}}.
\end{equation}

\subsection{Dihedral subgroups}
\label{b1:dihedral-subsection}

If the rotation subgroup of a dihedral group has order divisible
by \(p\), it has a unique fixed point in \(\PP^1(\F_q)\).
Normality makes this point fixed by the whole group, so
Proposition~\ref{b1:affine-quotient-equations} applies.
We now consider subgroups of split and nonsplit torus normalizers.
For \(\epsilon\in\{1,-1\}\), use \(x_+=u\) and \(x_-=x\),
respectively, where \(x\) is defined above. Put
\[
 S=\langle x_\epsilon\mapsto\zeta x_\epsilon,
 x_\epsilon\mapsto\eta/x_\epsilon\rangle,\qquad
 \operatorname{ord}(\zeta)=s\mid q-\epsilon,\qquad
 \eta^{q-\epsilon}=1,
\]
and set
\[
 b=\eta^s,\qquad e=\frac{q-\epsilon}{s},\qquad
 w=x_\epsilon^s+\frac{b}{x_\epsilon^s}.
\]
The parameter \(b\) records the reflection coset and need not
equal \(1\).
Write \(D_j(T,b)\) for the Dickson polynomials, defined by
\(D_0=2\), \(D_1=T\), and \(D_{j+1}=TD_j-bD_{j-1}\).
Thus \(D_j(y+b/y,b)=y^j+(b/y)^j\).

\begin{proposition}
\label{b1:dihedral-models}
Let \(U\) have admissible parameters \((d,S)\).
Then \(F^U=k(w,W)\), where \(W=v_0^2\) and
\begin{equation}
\label{b1:dihedral-dickson-model}
 W^\mu=\frac{\bigl(D_{e(q+\epsilon)}(w,b)-2\bigr)^{q+1}}
 {\bigl(D_e(w,b)-2\bigr)^{q^2+1}}.
\end{equation}
\end{proposition}

\begin{proof}
Since \(|S|=2s\), Proposition~\ref{b1:subgroup-realization}
gives \(\mu\) odd. The invariant \(w\) has degree \(2s\) as
a rational function of \(x_\epsilon\), so \(k(u)^S=k(w)\).
Put \(\chi=x_\epsilon^{q-\epsilon}\).
In the split case,
\(A(u)^{q-1}=\chi(\chi-1)^{q-1}\) and
\(B(u)=({\chi}^{q+1}-1)/(\chi-1)\).
In the nonsplit case, use \eqref{b1:torus-coordinate-factors}
and \(\Delta^{q(q-1)}=-1\).
Substitution in \(t_0=B(u)^{q+1}/A(u)^{q(q-1)}\) gives
\begin{equation}
\label{b1:dihedral-t0-expression}
 t_0=
 \begin{cases}
 \displaystyle\frac{({\chi}^{q+1}-1)^{q+1}}
 {{\chi}^q(\chi-1)^{q^2+1}},&\epsilon=1,\\[3mm]
 \displaystyle-\frac{\chi({\chi}^{q-1}-1)^{q+1}}
 {(\chi-1)^{q^2+1}},&\epsilon=-1.
 \end{cases}
\end{equation}
Since \(b^e=1\), we have
\(J=D_e(w,b)=\chi+{\chi}^{-1}\) and
\(D_{e(q+\epsilon)}(w,b)=D_{q+\epsilon}(J,1)\).
Substituting \(D_j(J,1)-2=({\chi}^j-1)^2/{\chi}^j\), for
\(j=1,q+\epsilon\), into the right-hand side of
\eqref{b1:dihedral-dickson-model} gives the square of the corresponding
expression in \eqref{b1:dihedral-t0-expression}, hence \(t_0^2\).
Proposition~\ref{b1:odd-kummer-generator} gives \(F^U=k(w,v_0)\),
and \(v_0=W^{(\mu+1)/2}/t_0\) gives \(k(w,v_0)=k(w,W)\).
\end{proof}

Theorem~\ref{b1:all-even-quotient-genera} gives genus
\((q^2-1-\nu(S))(\mu-1)/(4s)\), where \(\nu(S)=0\) for odd
\(q\) and \(\nu(S)=s\) for even \(q\).

For two models \eqref{b1:dihedral-dickson-model} with the same
\(d,\eta,\epsilon\) and rotation orders \(s\mid s'\), put
\(h=s'/s\). Then \(b'=b^h\) and \(w'=D_h(w,b)\), so the
quotient map of degree \(h\) is
\((w,W)\mapsto(D_h(w,b),W)\).

If \(\mu\mid q^2+1\), the same fixed field has the polynomial model
\begin{equation}
\label{b1:dihedral-polynomial-equation}
 F^U=k(w,Y),\qquad Y^\mu=D_{e(q+\epsilon)}(w,b)-2.
\end{equation}
Indeed, put \(P=D_{e(q+\epsilon)}(w,b)-2\),
\(Q_0=D_e(w,b)-2\), and \(H=WQ_0^{(q^2+1)/\mu}\).
Then \(H^\mu=P^{q+1}\). Since \(\gcd(\mu,q+1)=1\), choose
integers \(i,j\) with \(i(q+1)+j\mu=1\).
The function \(Y=H^iP^j\) satisfies \(Y^\mu=P\) and
\(Y^{q+1}=H\), so \(k(w,Y)=k(w,W)\).

\Needspace{7\baselineskip}
For the full pointwise and setwise pair stabilizers, put
\[
 S_\epsilon=\{x_\epsilon\mapsto\zeta x_\epsilon:
 \zeta^{q-\epsilon}=1\},\qquad
 \widehat S_\epsilon=\langle S_\epsilon,
 x_\epsilon\mapsto x_\epsilon^{-1}\rangle.
\]
The latter is the preceding dihedral group with
\(s=q-\epsilon\) and \(e=b=1\).
Put \(\chi=x_\epsilon^{q-\epsilon}\).
For odd \(q\), take \(t=(\chi+1)/(\chi-1)\), \(T=t^2\);
for even \(q\), take \(t=\chi/(\chi+1)\), \(T=t^2+t\).
Inversion sends \(t\) to \(-t\) and \(t+1\), respectively.
Thus \(k(u)^{S_\epsilon}=k(t)\) and
\(k(u)^{\widehat S_\epsilon}=k(T)\).
Define
\begin{equation}
\label{b1:conjugate-pair-polynomial}
 P_q(T)=
 \begin{cases}
 \displaystyle T^{(q+1)/2}(T-1)
   \left(\sum_{i=0}^{(q-3)/2}T^i\right)^{q+1},&q\text{ odd},\\[2mm]
 \displaystyle T\left(\sum_{i=0}^{a-1}T^{2^i-1}\right)^{q+1},&q=2^a,
 \end{cases}
\end{equation}
and
\begin{equation}
\label{b1:torus-split-function}
 Q_q(T)=
 \begin{cases}
 \displaystyle\frac{T^{(q+1)/2}(T^{(q-1)/2}+1)^{q+1}}
 {(T-1)^q},&q\text{ odd},\\[3mm]
 \displaystyle\frac{\left(1+\sum_{i=0}^{a-1}T^{2^i}\right)^{q+1}}
 {T^q},&q=2^a.
 \end{cases}
\end{equation}

\Needspace{8\baselineskip}
\begin{corollary}
\label{b1:conjugate-pair-quotients}
Let \(\widehat U_\epsilon\) have admissible parameters
\((d,\widehat S_\epsilon)\), and put
\(U_\epsilon=\widehat U_\epsilon\cap\pi^{-1}(S_\epsilon)\).
Then, with \(v=\epsilon v_0\),
\[
 F^{U_\epsilon}=k(t,v),\qquad
 F^{\widehat U_\epsilon}=k(T,v),
\]
where
\begin{equation}
\label{b1:conjugate-pair-models}
 v^\mu=\begin{cases}Q_q(T),&\epsilon=1,\\P_q(T),&\epsilon=-1.
 \end{cases}
\end{equation}
The degree-two quotient map is \((t,v)\mapsto(T,v)\).
\end{corollary}

\begin{proof}
The subgroup \(U_\epsilon\) has parameters \((d,S_\epsilon)\)
and index two in \(\widehat U_\epsilon\).
Admissibility and \(|\widehat S_\epsilon|=2(q-\epsilon)\)
give \(\mu\) odd.
For odd \(q\), we have \(\chi=(t+1)/(t-1)\) and
\[
 \chi-1=\frac{2}{t-1},\quad
 {\chi}^{q+1}-1=\frac{2(t^q+t)}{(t-1)^{q+1}},\quad
 (t+1)^{q-1}-(t-1)^{q-1}=-\frac{2(t^q-t)}{t^2-1}.
\]
Also \(2^{q(q-1)}=1\). For even \(q\),
\(\chi=t/(t+1)\) and \(\chi+1=1/(t+1)\).
Substitution in \eqref{b1:dihedral-t0-expression} gives
\begin{equation}
\label{b1:torus-pair-identities}
 t_0=
 \begin{cases}
 \displaystyle\epsilon\frac{(t^q+\epsilon t)^{q+1}}{(t^2-1)^q},
 &q\text{ odd},\\[2mm]
 \displaystyle\frac{(t^q+t+1)^{q+1}}{(t^2+t)^q},
 &q\text{ even},\ \epsilon=1,\\[2mm]
 \displaystyle\frac{(t^q+t)^{q+1}}{(t^2+t)^q},
 &q\text{ even},\ \epsilon=-1.
 \end{cases}
\end{equation}
Factoring \(t^q-t\) for odd \(q\), and using
\(t^q+t=\sum_{i=0}^{a-1}(t^2+t)^{2^i}\) for \(q=2^a\),
gives \eqref{b1:conjugate-pair-models}.
The fixed fields follow from Proposition~\ref{b1:odd-kummer-generator}.
\end{proof}

Theorem~\ref{b1:all-even-quotient-genera} gives
\begin{equation}
\label{b1:conjugate-pair-genera}
 g(F^{U_\epsilon})=\frac{(q+\epsilon)(\mu-1)}2,\qquad
 g(F^{\widehat U_\epsilon})=
 \begin{cases}
 \dfrac{(q+\epsilon)(\mu-1)}4,&q\text{ odd},\\[1mm]
 \dfrac{(q+\epsilon-1)(\mu-1)}4,&q\text{ even}.
 \end{cases}
\end{equation}

\Needspace{10\baselineskip}
\subsection{Subgroups over a quadratic subfield}
\label{b1:subfield-subsection}

Suppose that \(q=Q^2\), and put \(A_Q=u^Q-u\) and
\(B_Q=1+A_Q^{Q-1}\). Define
\[
 X=\frac{A_Q^{Q(Q-1)}}{B_Q^{Q+1}},\qquad
 f_Q(T)=\frac{(T^Q+T-1)^{Q^2+1}}{T^{Q(Q+1)}}.
\]
For odd \(Q\), also put
\(W=A_Q^{Q(Q-1)/2}/B_Q^{(Q+1)/2}\), so \(W^2=X\).

\begin{proposition}
\label{b1:square-subfield-models}
Let \(U_P\) have admissible parameters \((d,\PGL_2(Q))\).
Then
\begin{equation}
\label{b1:square-pgl-model}
 F^{U_P}=k(X,v_0),\qquad v_0^\mu=f_Q(X).
\end{equation}
For odd \(Q\), let \(U_L\) have parameters
\((d,\PSL_2(Q))\). Then
\begin{equation}
\label{b1:square-psl-model}
 F^{U_L}=k(W,v_0),\qquad v_0^\mu=f_Q(W^2).
\end{equation}
The degree-two quotient map is \((W,v_0)\mapsto(W^2,v_0)\).
\end{proposition}

\begin{proof}
The projective groups have even order, so \(\mu\) is odd by
Proposition~\ref{b1:subgroup-realization}.
The transformation formulas \eqref{b1:mobius-ab}, with \(q\)
replaced by \(Q\), show that \(X\) and, for odd \(Q\),
\(W\) are invariant under \(\PGL_2(Q)\) and
\(\PSL_2(Q)\), respectively.
Their degrees are the corresponding group orders
\(Q(Q^2-1)\) and \(Q(Q^2-1)/2\), so they generate the
fixed fields in \(k(u)\).

Put \(a=A_Q^{Q-1}\). Since \(B_Q=1+a\), we have
\(A(u)=A_QB_Q\) and
\[
 B(u)=1+a^{Q+1}B_Q^{Q^2-1}
 =B_Q^{Q(Q+1)}-a^QB_Q^{Q^2-1}-a^{Q^2}
 =-B_Q^{Q(Q+1)}(X^Q+X-1).
\]
Consequently,
\begin{equation}
\label{b1:square-subfield-invariant}
 t_0=\frac{B(u)^{q+1}}{A(u)^{q(q-1)}}=f_Q(X).
\end{equation}
Proposition~\ref{b1:odd-kummer-generator} now gives both fixed
fields and their equations. For odd \(Q\), uniqueness in
Proposition~\ref{b1:subgroup-realization} gives
\(U_L=U_P\cap\pi^{-1}(\PSL_2(Q))\).
The relation \(W^2=X\) gives the stated map, whose degree is
\([U_P:U_L]=2\).
\end{proof}

Since the projective images have \(Q^2-1\) elements of order \(p\),
Theorem~\ref{b1:all-even-quotient-genera} gives
\[
 g(F^{U_P})=\frac{Q(\mu-1)}2,\qquad
 g(F^{U_L})=Q(\mu-1)\quad(Q\text{ odd}).
\]

\Needspace{9\baselineskip}
\section*{Declaration on the use of AI}

The new family of curves was discovered through extensive searches
carried out using OpenAI's GPT-6 Astra.
Generative AI was also used to assist with mathematical calculations,
the development and checking of proofs, and the revision of the
manuscript. The authors take responsibility for the mathematical
content and the final text.

\Needspace{8\baselineskip}

\end{document}